\documentclass{amsart}
\usepackage{hyperref}

\usepackage{amssymb}

\usepackage[alphabetic]{amsrefs}

\usepackage{url}

\newcommand{\C}{\mathbb{C}}
\newcommand{\D}{\mathbb{D}}
\newcommand{\N}{\mathbb{N}}
\newcommand{\Q}{\mathbb{Q}}
\newcommand{\R}{\mathbb{R}}
\newcommand{\Z}{\mathbb{Z}}

\newcommand{\cA}{\mathcal{A}}
\newcommand{\cC}{\mathcal{C}}
\newcommand{\cH}{\mathcal{H}}
\newcommand{\cM}{\mathcal{M}}
\newcommand{\cR}{\mathcal{R}}
\newcommand{\cU}{\mathcal{U}}

\DeclareMathOperator{\re}{Re}
\DeclareMathOperator{\im}{Im}
\DeclareMathOperator{\poly}{Poly}
\DeclareMathOperator{\Int}{int}
\DeclareMathOperator{\crit}{Crit}

\DeclareMathOperator{\sector}{Sector}

\newtheorem{thm}{Theorem}[section]
\newtheorem{lem}[thm]{Lemma}

\newtheorem{cor}[thm]{Corollary}
\newtheorem{conj}[thm]{Conjecture}

\theoremstyle{definition}
\newtheorem*{defn}{Definition}

\theoremstyle{remark}
\newtheorem{rem}[thm]{Remark}

\newcommand{\cAp}{c_*}
\newcommand{\HAp}{\cH_*}
\newcommand{\cApAp}{c_{**}}
\newcommand{\HApAp}{\cH_{**}}

\usepackage{graphicx}

\title{Self-similarity for the tricorn}
\author{Hiroyuki Inou}
\thanks{Partially supported by JSPS KAKENHI Grant Numbers 22740105 and
26400115.}
\address{Department of Mathematics, Kyoto University\\
Kyoto 606-8502 Japan}

\begin{document}

\begin{abstract}
 We discuss self-similar property of the tricorn, the connectedness
 locus of the anti-holomorphic quadratic family. 
 As a direct consequence of the study on straightening maps by Kiwi and
 the author \cite{MR2970463}, we show that there are many homeomorphic
 copies of the Mandelbrot sets.
 With help of rigorous numerical computation,
 we also prove that the straightening map is not continuous for
 the candidate of a ``baby tricorn'' centered at the airplane,
 hence not a homeomorphism.
\end{abstract}

\maketitle

\section{Introduction}

In the study of dynamics of quadratic polynomials $Q_c(z)=z^2+c$, 
the most important object is the \emph{Mandelbrot set}:
\[
 \cM=\{c \in \C;\ K(Q_c)\mbox{ is connected}\},
\]
where $K(Q_c)=\{z \in \C;\ \{Q_c^n(z)\}_{n \ge 0}$ is bounded$\}$ is 
the \emph{filled Julia set} of $Q_c$.
It is well-known that ``baby Mandelbrot sets'', which are homeomorphic
to the Mandelbrot set itself, are dense in the boundary of the
Mandelbrot set \cite{MR816367} \cite{MR1765084}.

Analogously, one may consider the anti-holomorphic family of quadratic
polynomials:
\[
 f_c(z)= \bar{z}^2+c.
\]
The connectedness locus of this family is called the
\emph{tricorn} and we denote it by $\cM^*$ (Figure~\ref{fig:M and M^*}):
Since the second iterate $f_c^2$ is a (real-analytic) family
of holomorphic polynomials, one may regard the family as a special
family of holomorphic dynamics. 
In fact, Milnor \cite{MR1181083} numerically observed a ``little
tricorn'' in the real cubic family, 
and presented the tricorn as a prototype of such objects.
\begin{figure}[hbt]
 \centering
 \fbox{\includegraphics[width=5cm] {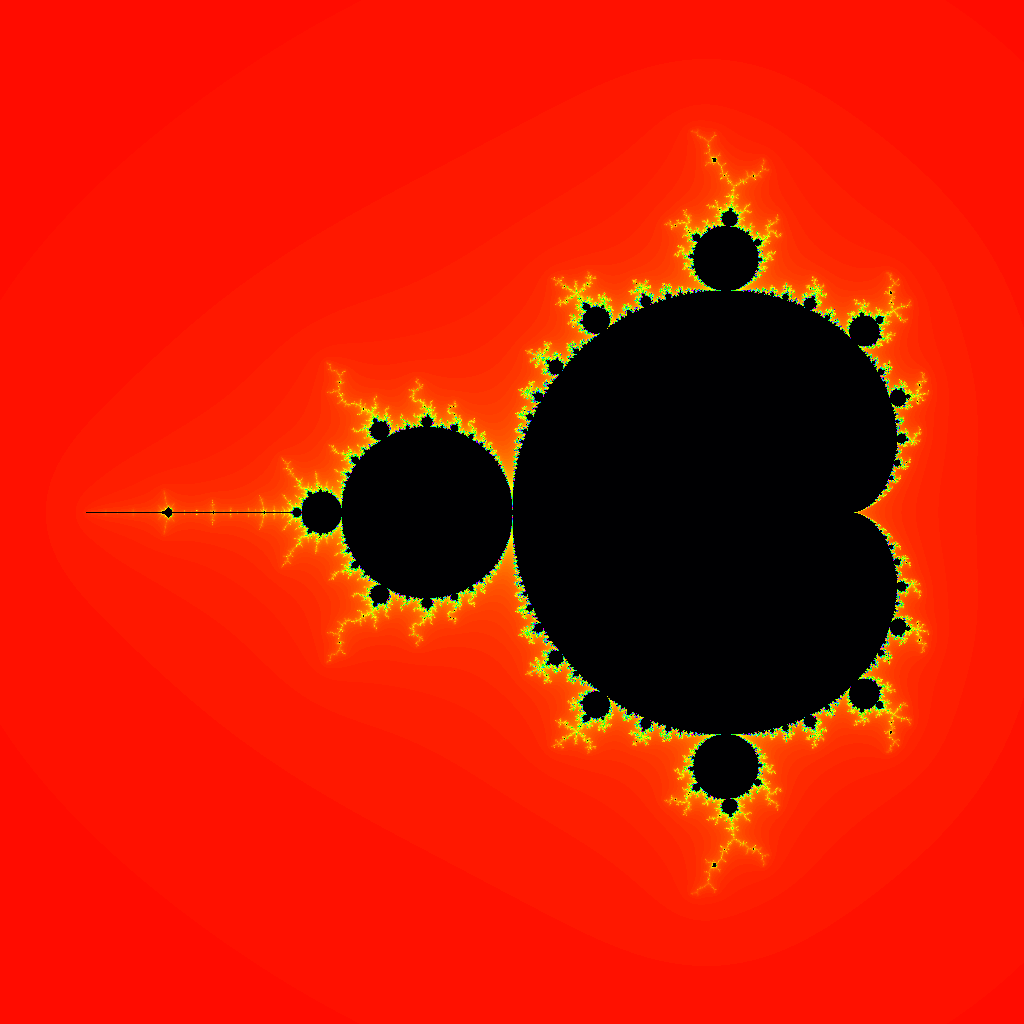}}\quad
 \fbox{\includegraphics[width=5cm] {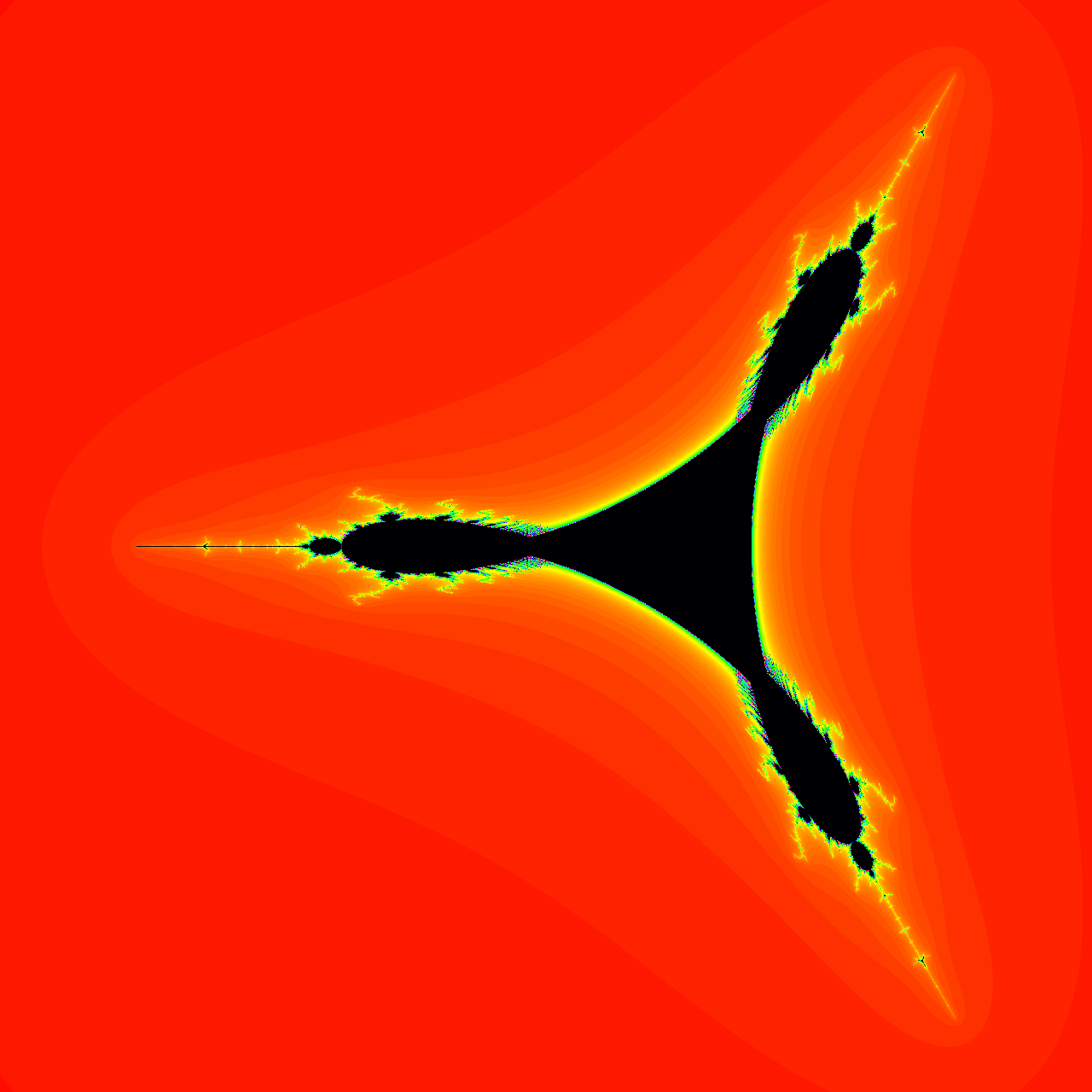}}
 \caption{The Mandelbrot set and the tricorn.}
 \label{fig:M and M^*}
\end{figure}

The tricorn was first named the \emph{Mandelbar set} and studied by
Crowe et al.\ \cite{MR1020441}. 
Nakane proved that the tricorn is connected by
constructing a real-analytic diffeomorphism between $\C \setminus \cM^*$
and $\C \setminus \overline{\D}$, where $\D$ is the unit disk,
using the B\"ottcher coordinate \cite{MR1235477}.
Mukherjee, Nakane and Schleicher studied its bifurcations
\cite{MR2020986} \cite{Mukherjee:2014ab}.
Numerically one can easily see that there are copies of the Mandelbrot
set, but some of those are ``blown up'' at the root; for example,
Crowe et al.\ \cite{MR1020441} observed that the main hyperbolic
component  (the one containing the origin) and each period two
hyperbolic component attached to it have a common boundary arc with a cusp.
On this common boundary arc, $f_c^2$ has a parabolic fixed point of
multiplier one. Hence this boundary arc should correspond to the root
point of the Mandelbrot set.
Moreover, this shows that the bifurcation locus is strictly bigger
than the boundary of the tricorn.

One of the most striking properties is that the tricorn is not 
arcwise connected.
It looks like that there is an ``umbilical cord'' accumulating to every
hyperbolic component of period greater than one, which connects the
component to the main hyperbolic component.
However, any reasonable picture of the ``umbilical cord'' of any
\emph{non-real} hyperbolic component of odd period oscillates and it
does not seem to converge to a point. 
Here \emph{non-real} stands for hyperbolic components intersecting the
real line and the lines symmetric to the real line by the rotational
symmetry of $\cM^*$. For real ones, umbilical cords are straight
segments and trivially land at a point.

Hubbard-Schleicher \cite{hubbard_multicorns_2014} (see also Nakane-Schleicher
\cite{MR1463779}) proved that in fact the umbilical cords
do not land under an assumption called OPPPP (Odd Period Prime Principal
Parabolic).
Note that it is conjectured that the Mandelbrot set is locally
connected, hence path connected.
In fact, it is known 
that any two Yoccoz parameters (i.e., at most finitely renormalizable
parameters) in the Mandelbrot set can be connected by an arc in the
Mandelbrot set \cite[Theorem~5.6]{Schleicher:aa}
(see also Petersen and Roesch \cite{MR2477422}).
Therefore, the umbilical cord always exists and converges to a point for
any hyperbolic component of the Mandelbrot set.

The non-landing phenomena of umbilical cords strongly suggests that
``little tricorns'' are not actually homeomorphic to the tricorn itself; 
although the umbilical cords on the real line land,
the corresponding umbilical cords in ``little tricorns'' do not lie in
the real line and its symmetries, hence they seems to oscillate.

The landing property of an umbilical cord is related to
real-analyticity of (a subtree of) the \emph{parabolic Hubbard tree}
in the dynamical plane for the landing parameter.
Therefore, if there is a periodic point with non-real multiplier on the
parabolic Hubbard tree, then it cannot be real-analytic, hence the
umbilical cord does not land \cite[Remark~6.4]{hubbard_multicorns_2014}.

In this article, we check this property for a specific ``little
tricorn'' using rigorous numerical computation and show the following:
\begin{thm}
 There is a baby tricorn-like set whose straightening map is not continuous.
\end{thm}

The precise definition of baby tricorn-like sets and their straightening
maps are given in Section~\ref{sec:ren}. Straightening maps can be
considered as a restriction of those defined for the family of quartic (or
biquadratic) polynomials in \cite{MR2970463}. See
Appendix~\ref{sec:app-streightening} for more details.

More precisely, what we prove is the following:
Let $\cAp \in \R$ be the airplane parameter, i.e., the critical point 0
is periodic of period three for $f_{\cAp}$. 
Consider the baby tricorn-like set $\cC(\cAp)$ centered at $\cAp$ 
and its straightening map $\chi_{\cAp}:\cC(\cAp) \to \cM^*$ (precise
definition is given in \S 4).  
Let $\cApAp=\chi^{-1}(\omega^2 \cAp)$. Then $\cApAp$ is the center of a
period 9 hyperbolic component $\HApAp \subset \cC(\cAp)$. 
We prove that the umbilical cord for $\HApAp$ does not land at a point
(see Figure~\ref{fig:H_**}) by rigorous numerical computation.
On the contrary, $\chi_{\cAp}(\HApAp)=\omega^2\HAp$ and the umbilical
cord for $\omega^2\HAp$ do land trivially (it is just a segment in
$\omega^2\R$).
We also prove that $\cM^* \cap \omega^2\R$ is contained in the image
$\chi_{\cAp}(\cC(\cAp))$, hence it follows that $\chi_{\cAp}$ is not
continuous.
\begin{figure}[ht]
 \centering
 \fbox{\includegraphics[width=5cm]
 {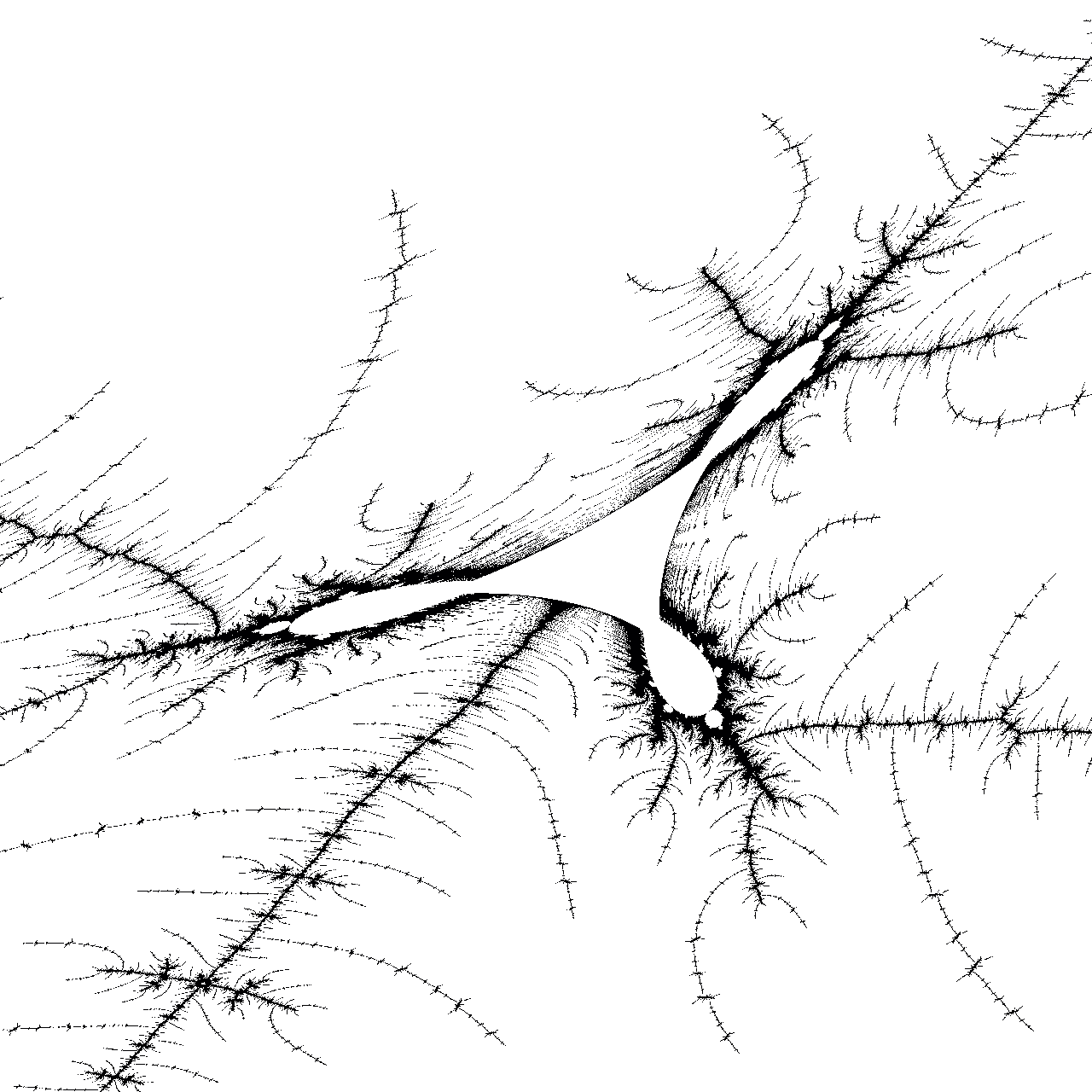}}\quad
 \fbox{\includegraphics[width=5cm]
 {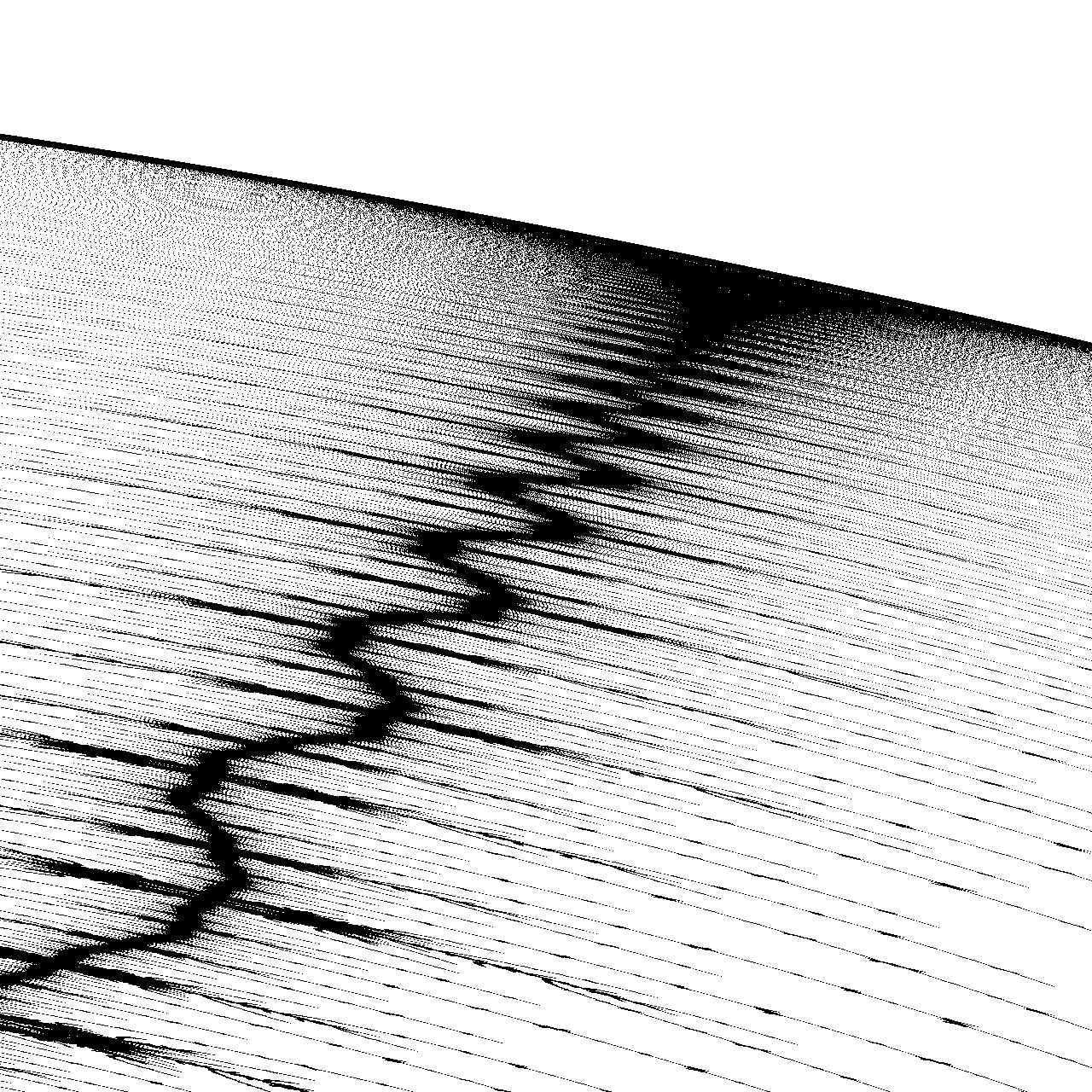}} 
 \caption{$\HApAp$ and its umbilical cord.}
 \label{fig:H_**}
\end{figure}

In the successive article with Mukherjee
\cite{straightening_discontinuity},
we give a non-computer-assisted proof of non-landingness of ``umbilical
cords'' for any non-real hyperbolic component in the family of
unicritical anti-holomorphic polynomials $\bar{z}^d+c$ of degree $d \ge
2$ (the \emph{multicorn family} ).
In particular, the theorem holds for any ``multicorn-like set'' of even
degree.
Note that when the degree $d$ is odd, this implies that there is no
landing ``umbilical cord'', hence we cannot conclude 
discontinuity of straightening maps from this result.

The author already proved discontinuity of straightening maps in
\cite{inou:2009aa} under more general setting.
However, its proof is by contradiction, so although it is constructive,
it is difficult to point out at which parameter a given straightening
map (or its inverse) is actually not continuous.
Moreover, the proof needs perturbations in two complex dimensional
parameter space. So it is quite difficult to get a numerical picture of
such discontinuity from this proof.
On the other hand, the tricorn is easy to draw and 
see what is happening at such a discontinuous parameter.

Numerical pictures also suggests that decorations attached to 
umbilical cords and hyperbolic components of double period 
also cause such discontinuity; indeed, one can also find 
a hyperbolic component of odd period whose boundary is \emph{not
accessible} from $\C \setminus \cM^*$ (Figure~\ref{fig:inaccessible}).

\begin{figure}[ht]
 \centering
 \fbox{\includegraphics[height=5cm]
 {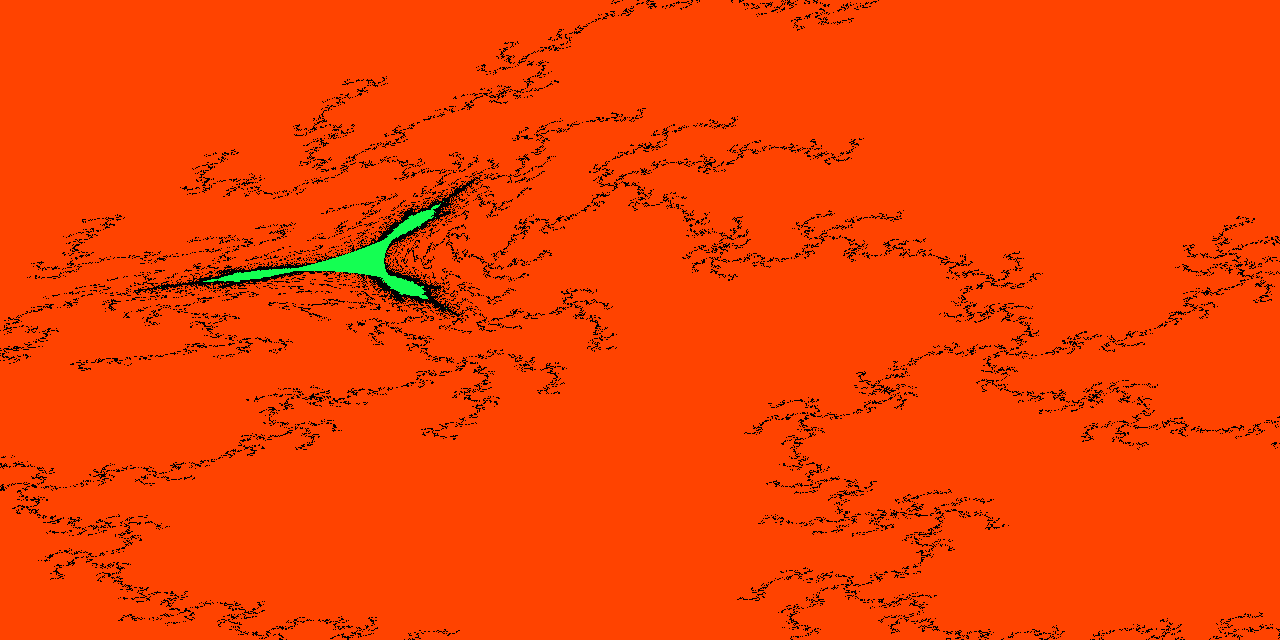}}
 \fbox{\includegraphics[width=5cm]
 {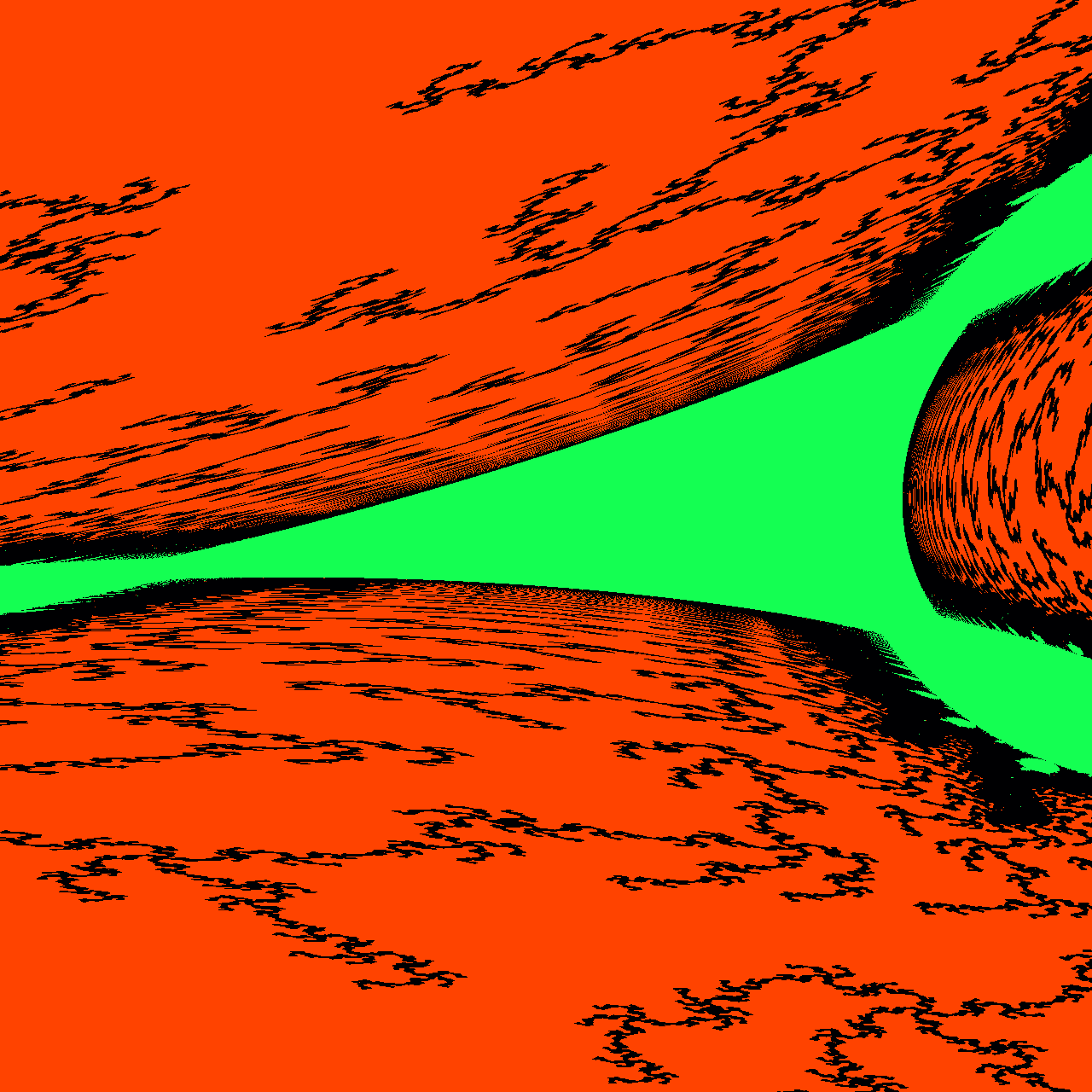}}
 \caption{An example of inaccessible hyperbolic component.}
 \label{fig:inaccessible}
\end{figure}

Hence it is very unlikely that any given two baby tricorn-like sets
have the same topological property for decorations of each pair of
corresponding hyperbolic components of odd period: 
\begin{conj}
 Any two baby tricorn-like sets are not homeomorphic unless they
 are symmetric.
\end{conj}

One may also consider baby tricorn-like sets in other families,
such as the real cubic family and the cubic antipodal-preserving rational
family (which is recently studied by Bonifant, Buff and Milnor \cite{antipode}).

The author would thank Shigehiro Ushiki and Zin Arai for helpful
discussions and comments. He would also thank Wolf Jung and Sabyasachi
Mukherjee for helpful comments.

\section{Tricorn}

Consider the following family of anti-holomorphic polynomials:
\[
 f_c(z) = \bar{z}^2+c.
\]
Observe that the second iterate $f_c^2(z) = (z^2+\bar{c})^2+c$ forms a
real-analytic two-parameter family of holomorphic dynamics.
Hence as in the case of holomorphic dynamics, one can consider the
\emph{filled Julia set} $K_c$ and the \emph{Julia set} for $f_c$
as follows:
\[
 K_c = \{z \in \C;\{f_c^n(z)\} \mbox{: bounded}\},\quad
 J_c = \partial K_c.
\]
The critical point of $f_c^2$ is $0$ and $f_c^{-1}(0)=\pm\sqrt{\bar{c}}$, hence 
$K_c$ and $J_c$ are connected if and only if the orbit of $0$ is
bounded (by $f_c$). The connectedness locus of this family is called the
\emph{tricorn} and denoted by $\cM^*$: 
\begin{align*}
 \cM^* &= \{c \in \C;\ K_c\mbox{: connected}\}  \\
 &= \{c \in \C;\ \{f_c^n(c)\}_{n\ge 0}\mbox{: bounded}\}.
\end{align*}

We denote the open disk of radius $r>0$ centered at $z$ by
$\D(r,z)$ and its closure by $\bar{\D}(r,z)$. We simply denote
them by $\D(r)$ and $\bar{\D}(r)$ when $z=0$ and we also denote $\D=\D(1)$.

\begin{lem}
 $\cM^* \subset \bar{\D}(2)$.
\end{lem}
See \cite[Theorem~2]{MR1020441}.

\begin{lem}
 For any $c \in \bar{\D}(2)$ (in particular for $c \in \cM^*$),
 we have $K_c \subset \bar{\D}(2)$.
\end{lem}

\begin{proof}
 Let $\varepsilon>0$. If $|z|>2+\varepsilon$,
 \begin{align*}
  |f_c(z)| &>|z|^2-|c| \ge |z|^2-2 = |z|(|z|-\frac{2}{|z|})
  >|z|(2+\varepsilon-1) \\
  &=|z|(1+\varepsilon).
 \end{align*}
 Hence $|f_c^n(z)| \to \infty$.
\end{proof}

For a periodic point $x$ of odd period $p$,
$\frac{\partial}{\partial \bar{z}}(f_c^p)(x)$ is not a conformal
invariant. 
Hence it is natural to define the \emph{multiplier} of
$x$ as $(f_c^{2p})'(z)$ (the multiplier as a periodic point of $f_c^2$).
Then it is easy to see the following:
\begin{lem}
 A periodic point of odd period has real non-negative multiplier.
\end{lem}
See \cite{MR2020986}.

The second iterate $f_c^2$ has four fixed points counted with
multiplicity.
Hence $f_c$ has at most four fixed points.
Indeed the following hold:
\begin{lem}
 \label{lem:num fixed pts}
 The number of fixed points for $f_c$, counted with multiplicity, is 
 \begin{itemize}
  \item 4 when $f_c$ has an attracting or parabolic fixed point 
	(equivalently, when $c$ lies in the closure of the main
	hyperbolic component). In this case there is no period two cycle
	for $f_c$.
  \item 2 otherwise. In this case $f_c$ has a unique period two cycle.
 \end{itemize}
\end{lem}
See \cite[Lemma~1]{MR1020441}.

Let $\cAp = -1.7548...$ be the airplane parameter, i.e., the unique
parameter $c<0$ satisfying $f_c^3(0)=0$.
\begin{lem}
 \label{lem:three per3 superattr}
 There are exactly three $c \in \C$ for which 0 is a periodic point of exact
 period 3 for $f_c$. They are $\cAp$, $\omega \cAp$ and $\omega^2 \cAp$
 where $\omega=\frac{-1+\sqrt{3}i}{2}$.
\end{lem}
More generally, the number of hyperbolic components of a given period is
counted in \cite{Mukherjee:2014ab}, but we give an elementary proof here
for completeness.

\begin{proof}
 Let $c_2 = f_c^3(0) = c^4+2c^2\bar{c}+c$ and let $s=c+\bar{c}$ and $t
 = c\bar{c}=|c|^2$. Then we have
 \begin{equation}
  \label{eqn:per3}
 \begin{aligned}
  2\re c_2 &=s^4+(1-4t)s^2+(1+2t)s+2t^2-2t,\\
  2\im c_2 &= (c-\bar{c})(s^3-(s-1)(1+2t)).
 \end{aligned}
 \end{equation}

 Now assume $c_2=0$.

 \noindent {\bfseries Case I: $c-\bar{c}=0$.} 
 Since $c$ is real, one can consider $Q_c(z)=z^2+c$
 instead of $f_c$. Hence it follows that $c=0$
 or $c=\cAp$.

 \noindent {\bfseries Case II: $c-\bar{c}\ne 0$.}
 Then we have $t=\frac{1}{2}\frac{s^3-s+1}{s-1}$. By
 substitution, we have
 \[
 (s^3-3s+3)(s^3-2s^2+s-1)=0.
 \]
 \noindent {\bfseries Case II-1: $s^3-3s+3=0$.}
 Then we have $t=1$ and $s<-2$,
 hence it follows $|c| = 1$ and $\re c < -1$ and this is a contradiction.

 \noindent {\bfseries Case II-2: $s^3-2s^2+s-1=0$.}
 Since $s$ is real,
 we have $s=-\cAp$.
 Then we have $t=\frac{1}{2}\frac{s^3+(s^3-2s^2)}{s-1}=s^2$ and 
 $c=\omega \cAp, \omega^2\cAp$.

 Therefore the solutions of \eqref{eqn:per3} are $c=0$, $\omega^k \cAp$
 ($k=0,1,2$). The case $c=0$ corresponds to the case that $0$ is a fixed
 point, hence the lemma follows. 
\end{proof}

\section{External rays and rational laminations}
\label{sec:ray-lamination}

When $K_c$ is connected, external rays for $f_c$ are defined in a usual
way; $f_c^2$ is a monic centered polynomial, hence the B\"ottcher
coordinate $\varphi_c:\C \setminus K_c \to \C \setminus \bar{\D}$ for
$f_c^2$ is defined, and indeed it satisfies 
\begin{equation}
 \label{eqn:Boettcher}
 \varphi(f_c(z)) = (\overline{\varphi(z)})^2.
\end{equation}
Let 
\[
 R_c(\theta) = \varphi^{-1}(\{re^{2\pi i \theta};\ r \in (1,\infty)\})
\]
be the \emph{external ray} of angle $\theta \in \R/\Z$ for $f_c$.
By the equation \eqref{eqn:Boettcher}, we have $f_c(R_c(\theta)) =
R_c(-2\theta)$.

\begin{defn}[Rational laminations]
 Let $d \in \Z\setminus\{0\}$ and denote $m_d(\theta)=d\theta$.
 An equivalence relation $\lambda$ on $\Q/\Z$ is called a
 \emph{$d$-invariant rational lamination} if the following hold:
 \begin{enumerate}
  \item $\lambda$ is closed in $(\Q/\Z)^2$.
  \item Every equivalence class is finite.
  \item For any equivalence class $A$, $m_d(A)$ is also an equivalence
	class.
  \item $m_d:A \to m_d(A)$ is consecutive-preserving, i.e., if
	$(\theta_1,\theta_2)$ is a component of $\R/\Z \setminus A$,
	Then $(d\theta_1,d\theta_2)$ is a component of $\R/\Z \setminus dA$.
  \item Equivalence classes are pairwise unlinked, i.e., for any
	$\lambda$-classes $A$ and $B$, $A$ is contained in a connected
	component of $\R/\Z \setminus B$.
 \end{enumerate}
 A \emph{real lamination} is an equivalence relation on $\R/\Z$
 satisfying the same condition replacing $\Q/\Z$ by $\R/\Z$.

 Assume $K_c$ is connected.
 The \emph{rational lamination} $\lambda(f_c)$ for $f_c$ is an
 equivalence relation 
 on $\Q/\Z$ such that $\theta$ and $\theta'$ are equivalent
 if and only if $R_c(\theta)$ and $R_c(\theta')$ lands at the same point.
\end{defn}

By the equation \eqref{eqn:Boettcher},
$\lambda(f_c)$ is a $(-2)$-invariant rational lamination.
More generally, for a holomorphic or anti-holomorphic polynomial of
degree $d \ge 2$ with connected Julia set, the rational lamination 
$\lambda(f)$ for $f$ is defined in the same way and it is 
$d$-invariant if $f$ is holomorphic, and $(-d)$-invariant if $f$ is
anti-holomorphic.
Note that the topological degree of a degree $d$ polynomial is $d$ if
it is holomorphic, and $-d$ if it is anti-holomorphic.

\begin{defn}[Unlinked classes and real extensions]
 Let $\lambda$ be a rational lamination.
 We say $\theta_1,\theta_2 \in (\R\setminus \Q)/\Z$ is
 \emph{$\lambda$-unlinked} if $\{\theta_1,\theta_2\}$ is unlinked with
 every $\lambda$-class. It is an equivalence relation and an equivalence
 class is called a \emph{$\lambda$-unlinked class}.

 Let $\bar{\lambda}$ denote the closure of $\lambda$ in $(\R/\Z)^2$
 and $\hat{\lambda}$ denote the smallest equivalence relation on $\R/\Z$
 containing $\bar{\lambda}$. We call $\hat{\lambda}$ the \emph{real
 extension} of $\lambda$.
\end{defn}
The real extension is a $d$-invariant real lamination \cite{MR1810538}. 
Its non-trivial equivalence classes are either $\lambda$-class or
a finite $\lambda$-unlinked class.

\begin{defn}[Critical elements]
 Let $\lambda$ be a rational lamination. For a $\lambda$-class or
 a $\lambda$-unlinked class $A$, $m_d:A \to m_d(A)$ is $\delta(A)$-to-one
 for some $\delta(A) \in \N$. We call $\delta(A)$ the \emph{degree}
 of $A$ and we say $A$ is \emph{critical} if $\delta(A)>1$.
 We call a critical element $A$ \emph{Julia critical element}
 if it is a $\lambda$-class or a finite $\lambda$-unlinked class, 
 and \emph{Fatou critical element} if it is an infinite
 $\lambda$-unlinked class.

 We denote by $\crit(\lambda)$ the set of critical elements for
 $\lambda$ and the critical orbit by
 \[
  CO(\lambda)=\{m_d^n(C);\ C \in \crit(\lambda),\ n\ge 0\}.
 \]
\end{defn}

\begin{defn}
 A $d$-invariant real lamination is \emph{hyperbolic} if for any
 $\lambda$-class $A$, $m_d:A \to A$ is one-to-one.

 A $d$-invariant rational lamination is \emph{hyperbolic} if its
 real extension is hyperbolic.
\end{defn}
In other words, a rational lamination is hyperbolic 
if all critical elements are Fatou critical elements.

The following is proved in \cite{MR1810538} (see \cite[Theorem~5.17]{MR2970463}).
\begin{lem}
 \label{lem:Kiwi-realization}
 Let $d \in \Z$ satisfy $|d| \ge 2$.
 For any hyperbolic $d$-invariant rational lamination, there exists a
 unique hyperbolic post-critically finite polynomial $f$ of degree
 $|d|$, holomorphic when $d>0$ and anti-holomorphic when $d<0$ such that
 $\lambda(f)=\lambda$.
\end{lem}
Since $f$ is hyperbolic, the Julia set $J(f)$ is locally connected
and $J(f) \simeq S^1/\hat{\lambda}$. Moreover, there is a one-to-one
correspondence between bounded Fatou components of $f$ and infinite
$\lambda$-unlinked classes.

\section{Renormalizations}
\label{sec:ren}

\begin{defn}[Polynomial-like maps]
 We say a map $g:U' \to U$ \emph{polynomial-like} if
 \begin{itemize}
  \item $U' \Subset U$ are topological disks in $\C$.
  \item $g:U' \to U$ is holomorphic or anti-holomorphic, and proper.
 \end{itemize}
 The \emph{filled Julia set} $K(g)$ and the \emph{Julia set}
 $J(g)$ is defined as follows:
 \[
  K(g) = \{z \in U';\ g^n(z) \in U'\ (n \in \N)\}, \quad
 J(g) = \partial K(g).
 \]
 
 A \emph{quadratic-like mapping} is a polynomial-like mapping of
 degree 2.
\end{defn}

\begin{defn}[Hybrid equivalence]
 Two polynomial-like mappings $f:U' \to U$ and $g:V' \to V$ are
 \emph{hybrid equivalent} if there exists a quasiconformal homeomorphism
 $\varphi:U'' \to V''$ between neighborhoods $U''$ and $V''$ of $K(f)$
 and $K(g)$ respectively, such that $\varphi\circ f = g \circ \varphi$
 where both sides are defined and $\bar{\partial}\varphi \equiv 0$
 almost everywhere in $K(f)$. 
\end{defn}

\begin{thm}[Straightening theorem]
 Any polynomial-like mappings $g:U' \to U$ is hybrid equivalent to a
 holomorphic or anti-holomorphic polynomial $P$ of the same degree.
 Moreover, if $K(g)$ is connected, then $P$ is unique up to affine conjugacy.
\end{thm}
See \cite{MR816367} and \cite{hubbard_multicorns_2014}.

\begin{defn}[Renormalizations and straightenings]
 We say $f_c$ is \emph{renormalizable} if
 there exist $U'$, $U$ and $n>1$ such that $f_c^n:U' \to U$ is a
 quadratic-like mapping with connected filled Julia set.

 Such a mapping $f_c^n:U' \to U$ is called a \emph{renormalization}
 of $f_c$ and $n$ is called the \emph{period} of it.

 By the straightening theorem, there exists a monic centered
 holomorphic or anti-holomorphic quadratic polynomial $P$ hybrid
 equivalent to $f_c^n:U' \to U$, up to affine conjugacy.
 We call $P$ the \emph{straightening} of it.
\end{defn}


Take $c_0 \in \cM^*$ such that $0$ is periodic of period $n>1$ by
$f_{c_0}$. We call such $c_0$ simply a \emph{center}
(precisely speaking, it is the center of a hyperbolic component of $\Int
\cM^*$) and $n$ the \emph{period} for $c_0$.
Let $\lambda_0 = \lambda(f_{c_0})$ be the rational lamination for
$f_{c_0}$. Define the combinatorial renormalization locus
$\cC(c_0)$ and the
renormalization locus $\cR(c_0)$ with combinatorics $\lambda_0$ as
follows: Let
\[
 \cC(c_0) = \{c \in \cM^*;\ \lambda(f_c) \supset \lambda_0\}
\]
(recall that a relation on $\Q/\Z$ is a subset of $(\Q/\Z)^2$).
Let $c \in \cC(c_0)$. By definition, the external rays of
$\lambda_0$-equivalent angles for $f_c$ 
land at the same point. Hence those rays divide $K_c$ into
``\emph{fibers}'' 
(see Appendix~\ref{sec:app-tree-renorm} for the precise definition). 
Let $K$ be the fiber containing the critical point. Then
$f_c^n(K)=K$. We say $f_c$ is \emph{$c_0$-renormalizable} if
there exists a (holomorphic or anti-holomorphic) quadratic-like
restriction $f_c^n:U' \to U$ such that the filled 
Julia set is equal to $K$. Let
\[
 \cR(c_0) = \{c \in \cC(c_0);\ c_0\mbox{-renormalizable}\}.
\]
We call such a renormalization a \emph{$c_0$-renormalization}.
(See the definition of \emph{$\lambda_0$-renormalization} in
\cite{MR2970463} for the precise definition.)
We call $n$ the \emph{renormalization period}.

In the following, we fix such $c_0$ (or a hyperbolic rational lamination
$\lambda_0=\lambda(f_{c_0})$).
For $c \in \cR(c_0)$, let $P$ be the
straightening of a $c_0$-renormalization of $f_c$.
By the straightening theorem, $P$ is well-defined.
When the renormalization period $n$ is even, 
then $c_0$-renormalization is holomorphic.
Hence $P=Q_{c'}$ for a unique $c' \in \cM$.
When $n$ is odd, the $c_0$-renormalization is anti-holomorphic,
so $P = f_{c'}$ for some $c' \in \cM^*$.
In either case, we denote $c'$ by $\chi_{c_0}(c)$ and let
\[
 \cM(c_0) =
 \begin{cases}
  \cM & n\mbox{: even,}\\
  \cM^* & n\mbox{: odd.}
 \end{cases}
\]
Although we have three choices for $c'$ when the period is odd,
we can naturally choose one
by fixing an \emph{internal marking} (for the precise definition, see
Appendix~\ref{sec:app-streightening} and \cite{MR2970463}).

\begin{defn}
 We call the map $\chi_{c_0}:\cR(c_0) \to \cM(c_0)$
 above the \emph{straightening map} for $c_0$.

 We call $\cC(c_0)$ a \emph{baby Mandelbrot-like set} if the
 renormalization period $n$ is even,
 and a \emph{baby tricorn-like set} if $n$ is odd.
 If $\chi_{c_0}$ can be extended to a homeomorphism between
 $\cC(c_0)$ and $\cM(c_0)$, then we call $\cC(c_0)$
 a \emph{baby Mandelbrot set} (resp.\ a \emph{baby tricorn})
 if the renormalization period is even (resp.\ odd).
\end{defn}
As noted above, 
if the period $n$ is odd,
then $\omega\chi_{c_0}$ and $\omega^2\chi_{c_0}$ are also straightening
maps (with different internal markings).
In this paper, we always choose and fix one of those.

\begin{defn}
 We call a center $c_0 \in \cM^*$ (or $f_{c_0}$, $\lambda_0$) 
 \emph{primitive} if the closures of Fatou components for $f_{c_0}$ are
 mutually disjoint.
\end{defn}

By applying results in \cite{MR2970463}, we have the following:

\begin{thm}[Injectivity]
 \label{thm:IH-inj}
 $\chi_{c_0}:\cR(c_0) \to \cM(c_0)$ is injective.
\end{thm}

\begin{thm}[Onto hyperbolicity]
 \label{thm:IH-onto hyp}
 The image of the straightening map $\chi_{c_0}(\cR(c_0))$ contains all 
 hyperbolic maps in $\cM(c_0)$.
\end{thm}

\begin{thm}[Compactness]
 \label{thm:IH-cpt}
 If $c_0$ is primitive,
 then we have $\cC(c_0) = \cR(c_0)$ and it is compact.
\end{thm}

\begin{thm}[Surjectivity]
 \label{thm:IH-onto}
 If $c_0$ is primitive and the renormalization period is even,
 then the corresponding straightening map $\chi_{c_0}:\cR(c_0) \to \cM$ 
 is a homeomorphism.
\end{thm}

Theorem~\ref{thm:IH-onto} shows that there are baby Mandelbrot sets in
the tricorn. For example, there are infinitely many primitive centers of
even period in $\cM^* \cap \R= \cM \cap \R$.
Surjectivity of straightening maps for (primitive) tricorn-like sets
is still open.

How these theorems are obtained from more general results in
\cite{MR2970463} is explained in Appendix~\ref{sec:app-streightening}.

\begin{rem}
 Except for the case of Theorem~\ref{thm:IH-onto}, 
 we do not know if $\cC(c_0)$ or $\cR(c_0)$ are connected, even if we
 assume that $c_0$ is primitive.

 As in the case of the Mandelbrot set,
 it is natural to consider $\cC(c_0)$ as a \emph{fiber}
 by cutting parameter rays landing at the same point.
 Since each fiber is connected, this argument should prove $\cC(c_0)$ is
 connected.  
 However, one should be aware that the parameter rays accumulating to
 hyperbolic components of odd period $p>1$ does not land at a point
 \cite{inou:2014aa}.
 For example, the parameter rays of angle $3/7$ and $4/7$ for the
 Mandelbrot set land at the root of the airplane component and they divide
 the parameter plane into two.
 On the other hand, those rays for $\cM^*$ do not land
 and that their accumulation sets are disjoint (see
 Figure~\ref{fig:parameter ray}). 
 But they, together with the airplane component, still divide the plane
 into two components. 
 \begin{figure}[htb]
  \centering
  \fbox{\includegraphics[width=5cm]{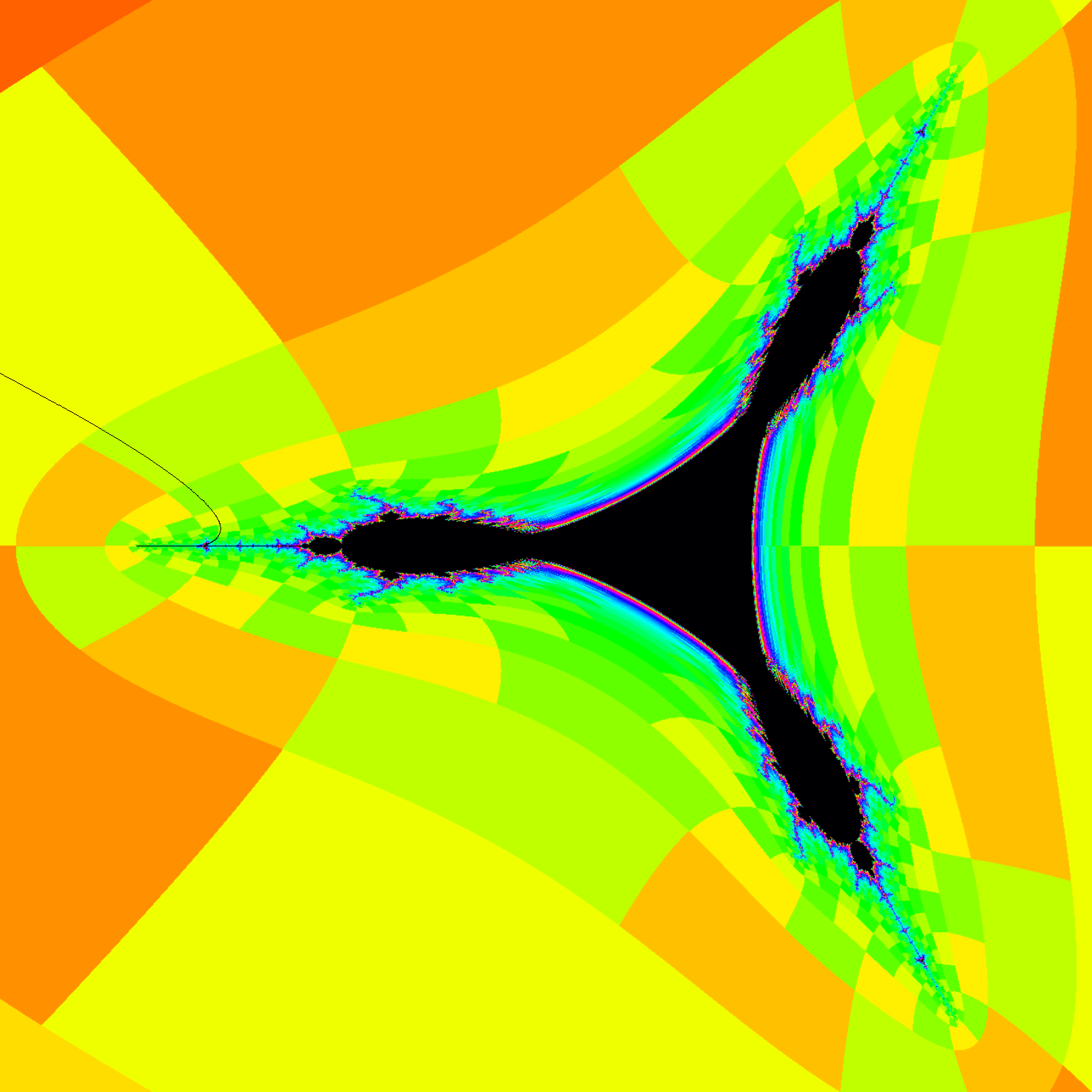}}\quad
  \fbox{\includegraphics[width=5cm]{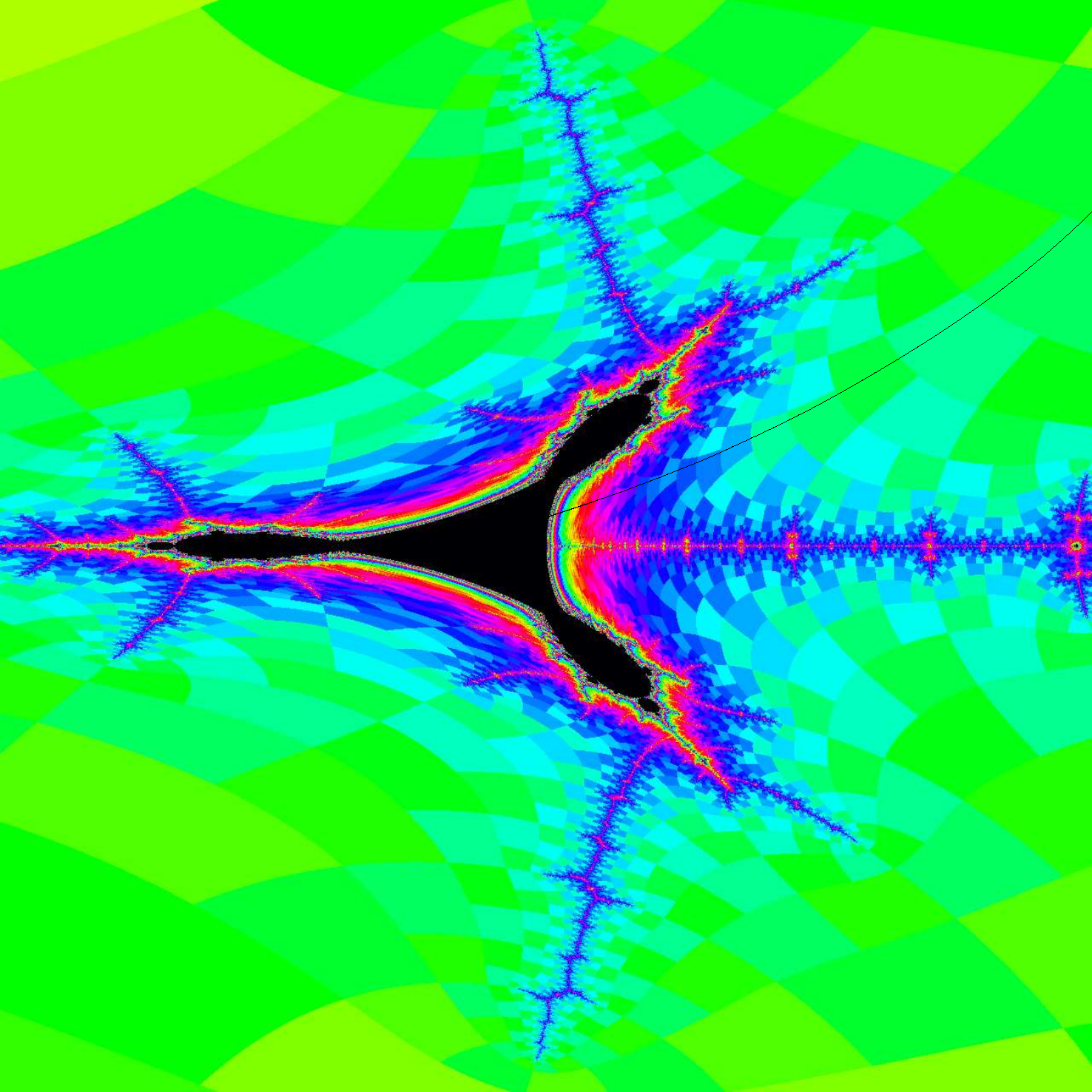}}\\
  \fbox{\includegraphics[width=5cm]{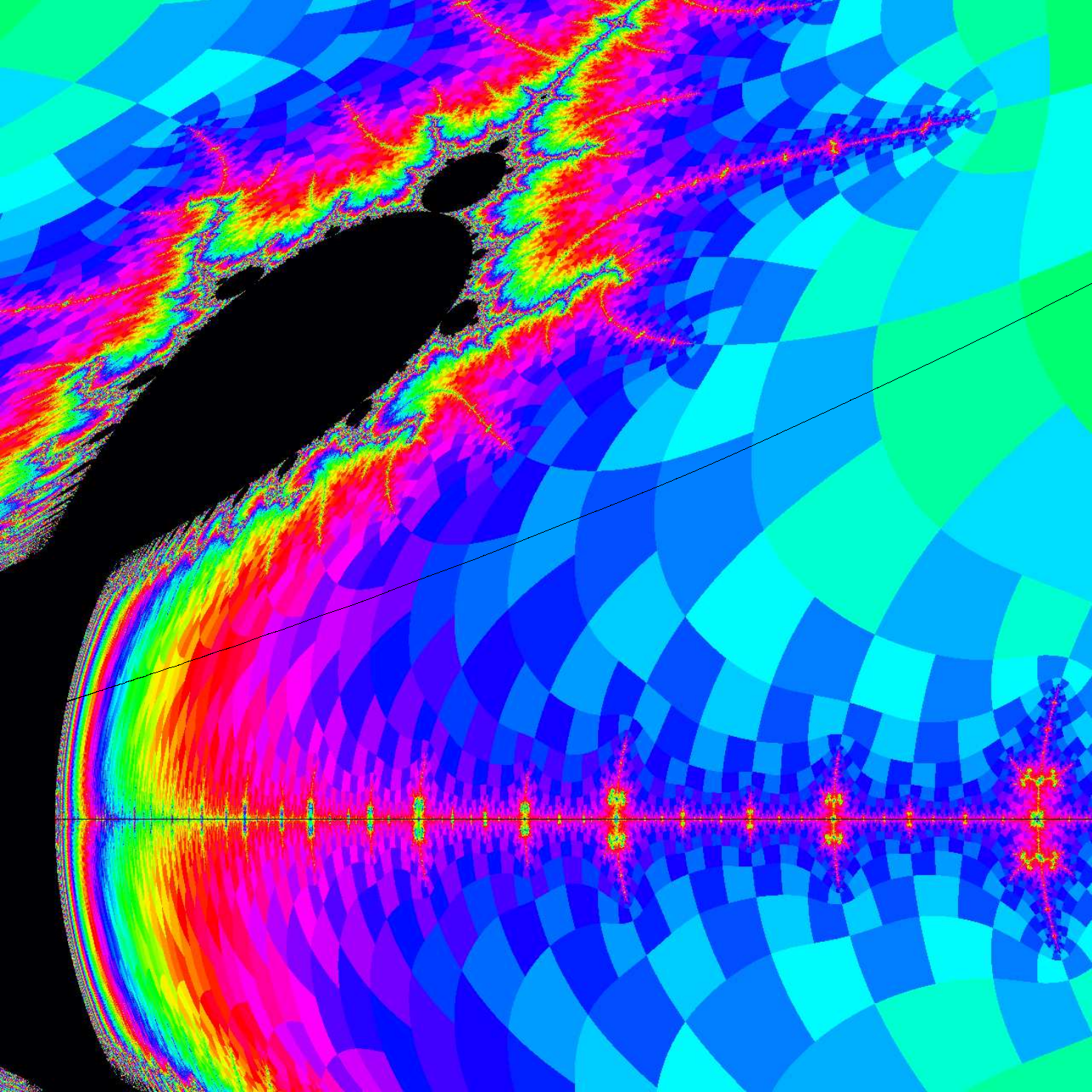}}\quad
  \fbox{\includegraphics[width=5cm]{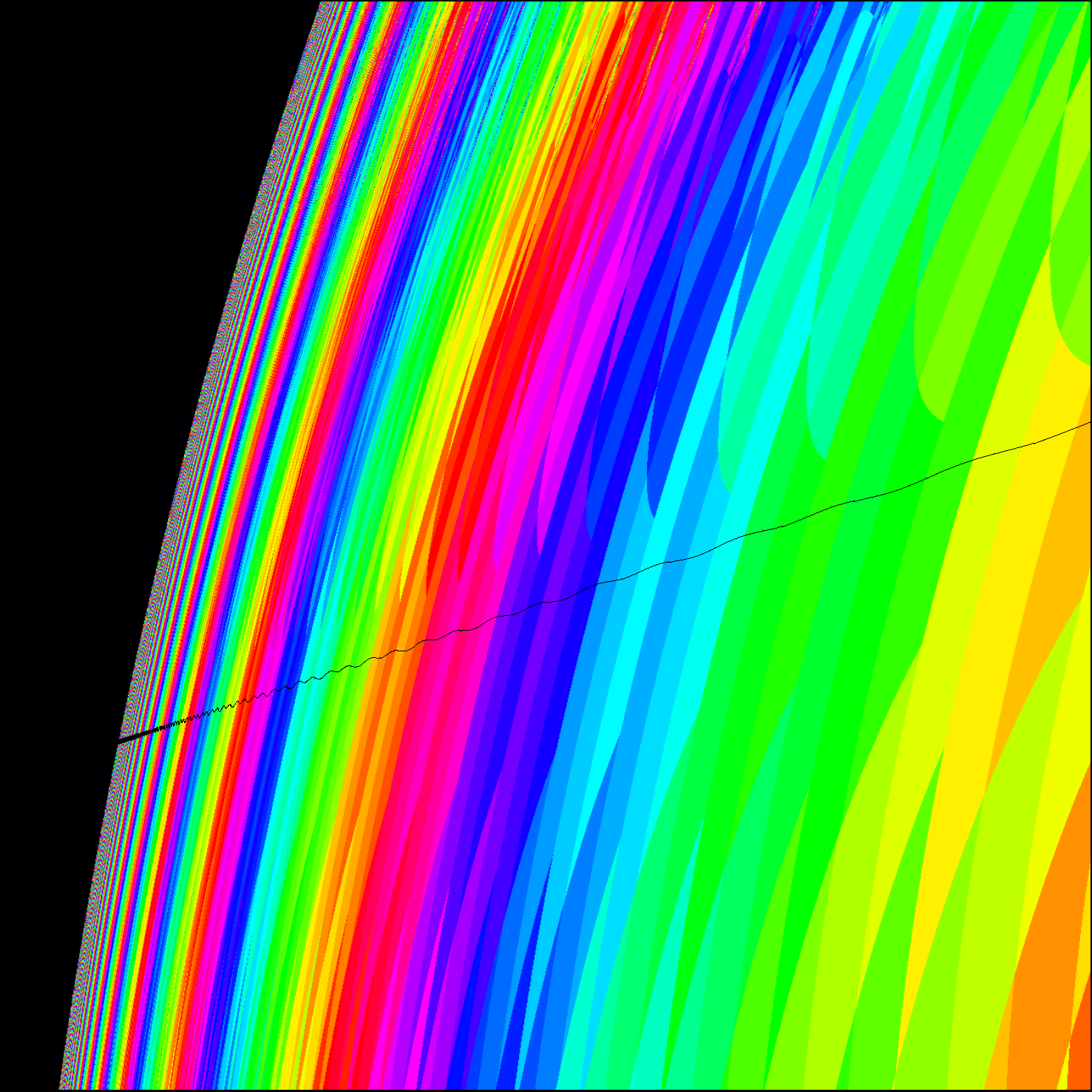}}\\
  \fbox{\includegraphics[width=5cm]{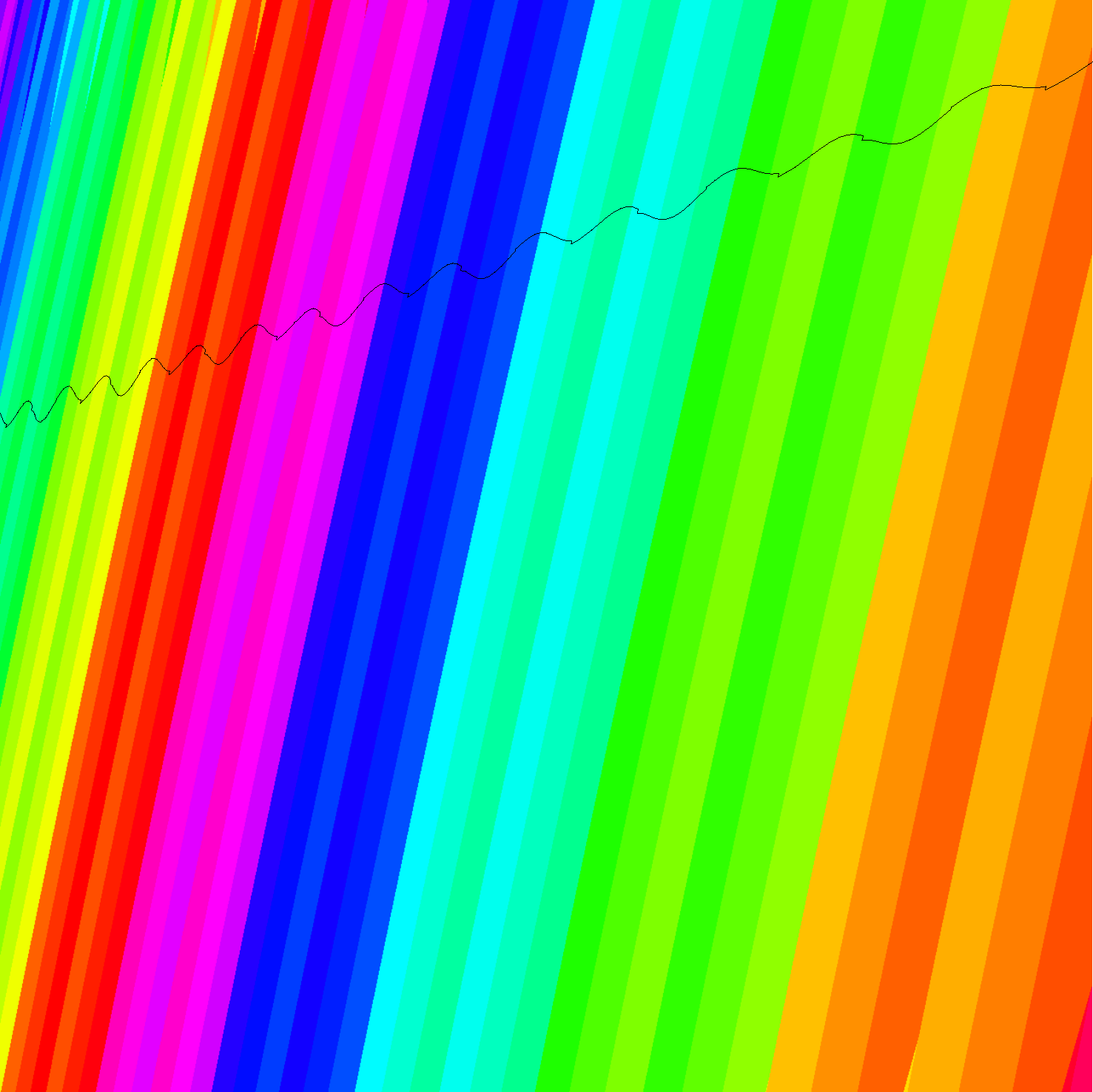}}
  \caption{The parameter ray of angle 3/7.}
  \label{fig:parameter ray}
 \end{figure}
\end{rem}

\section{Baby tricorn-like sets}
\label{sec:baby-tricorn-like}

\begin{lem}
 \label{lem:odd-primitive}
 If the renormalization period is odd, then $c_0$ is primitive.
\end{lem}

\begin{proof}
 If there are two distinct Fatou components $U_1$, $U_2$ for $f_{c_0}$
 such that $\overline{U_1} \cap \overline{U_2}$ is nonempty,
 then by taking the forward images, we may assume that $U_1$ and $U_2$
 are periodic. Then the intersection consists of a repelling periodic
 point $x$ of period $k$, which is less than the renormalization period
 $n$.
 
 Since $x$ is the ``root'' of $U_1$ (i.e., $x$ disconnects the filled
 Julia set) and the period $n>1$ is odd, we have $k=n$ by
 \cite[Corollary~4.2]{MR2020986}, hence we have a
 contradiction.
\end{proof}

By Theorem~\ref{thm:IH-cpt}, we have the following:
\begin{cor}
 \label{cor:cpt}
 If the renormalization period is odd, then
 $\cC(c_0)=\cR(c_0)$ and it is compact.
\end{cor}

By a similar argument as in \cite[\S 9]{MR2970463},
we have surjectivity onto the real connectedness locus.
\begin{cor}
 \label{cor:real surjectivity}
 If the renormalization period is odd, then
 the image $\chi_{c_0}(\cR(c_0))$ of the straightening map contains 
 the real connectedness locus $\cM^* \cap \R=[-2,1/4]$.
\end{cor}

\begin{proof}
 First note that $f_c^2=Q_c^2$ for $c \in \R$, hence hyperbolic maps are
 dense on the real line by Graczyk-\'{S}wi\c{a}tek
 \cite{MR1469316} and Lyubich \cite{MR1459261}.

 By Theorem~\ref{thm:IH-onto hyp}, the image contains those
 real hyperbolic parameters.
 Therefore for $c' \in \R \cap \cM^*$,
 we can take a sequence $c_n \in \cR(c_0)$ satisfying $\chi_{c_0}(c_n) \in
 \R$ and $\lim_{n \to \infty} \chi_{c_0}(c_n) \to c'$.
 Since $\cR(c_0)$ is compact, we may assume that $c_n \to c$ for some
 $c \in \cR(c_0)$.
 Let $c''=\chi_{c_0}(c)$. Then $f_{c'}$ and $f_{c''}$ are
 quasiconformally conjugate \cite[\S II.7 Lemma]{MR816367}. 
 Since $f_{c'}$ and $f_{c''}$ are conformally conjugate outside the
 filled Julia sets,
 there exists some $\tilde{c} \in
 \cR(c_0)$ such that $\chi_{c_0}(\tilde{c})=c'$ by
 \cite[Lemma~9.2]{MR2970463}.
\end{proof}

By Theorem~\ref{thm:IH-inj}, Lemma~\ref{lem:num fixed pts} and
Lemma~\ref{lem:three per3 superattr}, the following hold:
\begin{lem}
 For $c \in \cR(c_0)$, a renormalization $f_c^n:U' \to U$ of $f_c$ has
 either 
 \begin{itemize}
  \item four fixed points when the renormalization has an attracting or
	parabolic fixed point and no period two cycle;
  \item two fixed points and one period two cycle otherwise.
 \end{itemize}
\end{lem}

\begin{lem}
 There are exactly three polynomials in $\cC(c_0)$ for which
 $c_0$-renormalization has period three critical point.
\end{lem}

\section{Umbilical cords and path connectivity}
\label{sec:umbilical cords}

Hubbard-Schleicher \cite{hubbard_multicorns_2014} (see also Nakane-Schleicher
\cite{MR1463779}) showed that the tricorn is not path
connected.
Each hyperbolic component of odd period $p>1$ seems to have
an ``umbilical cord'', which connects it to the main hyperbolic component
(the one containing the origin).
They proved that such an umbilical cord, if exists, oscillates and does
not converge to a point, under an assumption called OPPPP
(Odd Period Prime Principal Parabolic). 

In this section, we briefly review their results and apply them to our
specific case.

\subsection{Ecalle cylinders, equator and Ecalle heights}

\begin{lem}
 Assume $f_c$ has a simple parabolic periodic point $z_0$ of odd period
 $p$ and let $V$ be the parabolic basin of $z_0$.
 Then there exists an open set $U \subset V$ with $z_0 \in \partial U$
 and a holomorphic map  $\varphi:U \cap V \to \C$ such that
 \begin{itemize}
  \item $\varphi(U \cap V)$ contains the half plane $\{\re w > 0\}$;
  \item $\varphi \circ f_c^k \circ \varphi^{-1}(w)= \bar{w}+\frac{1}{2}$.
 \end{itemize}
 The map $\varphi$ is unique up to post-composition by translation by
 real constant.
\end{lem}
See \cite[Lemma~2.3]{hubbard_multicorns_2014}.

The map $\varphi$ is called an \emph{attracting (incoming) Fatou
coordinate} at $z_0$.
It induces an isomorphism from the quotient $(V \cap
U)/(f_c^{2k})$ to the cylinder $\{\re w > 0\}/(w \mapsto w+1) = \C/\Z$.
The quotient space is called the \emph{attracting (incoming) Ecalle
cylinder} for $z_0$.
The set of fixed points by the action of $f_c$ on the Ecalle cylinder
forms a closed geodesic corresponding to $\R/\Z$.
We call it the \emph{equator}.
For $w \in V$, there exists some $k\ge 0$ such that $f_c^{2k}(w) \in U$.
The \emph{Ecalle height} of $w$ is defined by 
$\im \varphi(f_c^{2k}(w))$.
The Ecalle height of the critical point is called the \emph{critical
Ecalle height}.

Similarly, by replacing $f_c$ by $f_c^{-1}$, one can consider an 
\emph{repelling (outgoing) Fatou coordinate} and the \emph{repelling
Ecalle cylinder} for $z_0$. 

\subsection{Parabolic arcs}

\begin{thm}
 The boundary of a hyperbolic component of odd period 
 consists of three real-analytic curves, each of which is parametrized by
 the critical Ecalle height. 
\end{thm}
Those arcs are called \emph{parabolic arcs}.
A parabolic arc is called a \emph{root arc} if
the parabolic orbit of $f_c$ disconnects the Julia set for any parameter
$c$ on the arc, otherwise, it is called a \emph{co-root} arc.
Endpoints of parabolic arcs are called \emph{cusps}.
Each hyperbolic component of odd period $p \ge 3$ has one root arc and
two co-root arcs. 
By definition, if the period is one (i.e., for the main hyperbolic
component), all parabolic arcs are symmetric and co-root.

\subsection{Parabolic Hubbard tree}

\begin{defn}[Characteristic parabolic point]
 Let $f_c$ have a parabolic periodic point.
 The \emph{characteristic parabolic point} is the unique parabolic
 point on the boundary of the critical value Fatou component.
\end{defn}

\begin{defn}[Principal parabolic]
 We say $f_c$ (or simply $c$) is \emph{principal parabolic} if $f_c$
 has a simple parabolic orbit of period greater than one 
 and each point in the orbit has at least two periodic landing rays.
\end{defn}

If $f_c$ is principal parabolic and the period of the parabolic orbit is
odd, then $f_c$ is on a root-arc or its endpoints.

\begin{defn}[Parabolic Hubbard tree]
 Assume $f_c$ is principal parabolic.
 A \emph{parabolic Hubbard tree}, or a \emph{parabolic tree} for
 simplicity, of $f_c$ is a minimal tree contained in $K_c$ and
 connecting the parabolic orbit and the critical orbit, so that the
 intersection with the critical value Fatou component is a simple
 $f_c^p$-invariant curve connecting the critical value to the parabolic
 point on the boundary,
 and its intersection with any
 Fatou component is the backward image of the curve in the critical
 value component.
\end{defn}

A parabolic tree is not unique; it is unique in the Julia set,
but one can take any essential closed curve in the Ecalle cylinder
passing the critical value and then lift it and take forward and backward
images) to construct a parabolic Hubbard tree.
Hence any two parabolic Hubbard trees are homotopic rel $J_c$.
Sometimes it is called a \emph{loose} parabolic tree to emphasize
the non-uniqueness.

\begin{lem}
 \label{lem:arc converging parabolic}
 Let $f_c$ be principal parabolic with odd period $p>1$ and
 let $x$ be the characteristic parabolic point of period $p$ for $f_c$.
 Then 
 \begin{enumerate}
  \item $x$ is the landing point of exactly two rays of period $2p$.
  \item \label{item:endpoint} 
	Any parabolic tree of $f_c$ contains an arc such that 
	it converges to $x$ and does not intersect the critical value Fatou
	component.
 \end{enumerate}
\end{lem}

One can find a proof of (\ref{item:endpoint}) for holomorphic quadratic
polynomials in Eberlein's thesis \cite{eberlein_rational_1999}, which
can be applied to our case. Hence we give an outline of the proof here.
\begin{proof}
 For the first part, see \cite[Corollary~4.2]{MR2020986}. 

 For the second, observe that all the end points of a given parabolic tree
 are contained in $\{0,f_c(0),\dots,f_c^{p-1}(0)\}$.
 Therefore, $f_c^k(0)$ is an endpoint for some $0<k<p$.
 Since a parabolic tree is forward invariant and $f_c^{k-1}$
 is one-to-one near $f_c(0)$, $f_c(0)$ is also an endpoint.

 By forward invariance, the intersection with a given parabolic tree and the
 boundary of the critical value component is a fixed point of $f_c^p$
 which disconnects the Julia set. Hence it follows that it must be equal
 to $x$. 
\end{proof}
By primitivity, the intersection of such an arc and the Julia set is a
Cantor set.
By the lemma, a parabolic tree corresponds to a simple closed curve
in the repelling Ecalle cylinder for $x$, which we call a \emph{Hubbard loop}. 

\subsection{Real analyticity of parabolic trees and path
connectivity of the tricorn}

Most part of the proof of non-path connectivity of multicorns in
\cite{hubbard_multicorns_2014} works not only for OPPPP case,
but also in general.
We apply their results only to the case explained in the
introduction; i.e., for the parameter of critical Ecalle height zero in
the root arc of $\HApAp$.
Although some proofs become simpler for this case, we state and prove in
general.

The following follows by the same argument of the first part of the
proof of \cite[Lemma~5.8]{hubbard_multicorns_2014}:
\begin{lem}
 \label{lem:analytic arc}
 Let $c \in \cM^*$ be on the parabolic root arc of a hyperbolic
 component of odd period $p$. If $K(f_c)$ contains an analytic arc $\gamma$
 connecting two bounded Fatou components, 
 then there exists an analytic arc $\gamma'$
 such that 
 \begin{enumerate}
  \item $\displaystyle \gamma' \subset \bigcup_{n \ge 0} f_c^n(\gamma)$.
  \item there exists a parabolic tree which contains $\gamma'$.
  \item $\gamma'$ is invariant by $f_c^p$.
  \item $\gamma'$ contains the characteristic parabolic point.
 \end{enumerate}
\end{lem}

\begin{cor}
 \label{cor:analytic tree}
 Under the assumptions and the notations in the lemma, let $q|p$ be
 the smallest integer such that $\gamma'$ is invariant by $f_c^q$.
 Then $q$ is strictly smaller than $p$.
 Moreover if $q>1$, then $f_c$ is renormalizable of period $q$.
 The period $q$ renormalization has an unbranched and analytic parabolic
 tree contained in $\gamma'$.
\end{cor}

\begin{proof}
 If $\gamma'$ is a loose parabolic Hubbard tree for $f_c$, then $q=1$
 and we are done.

 Otherwise, we can apply Theorem~\ref{thm:tree-renorm} 
 and it follows that $q>1$ and $f_c$ is renormalizable of period $q$.
 The renormalization has a loose parabolic Hubbard tree contained in an
 analytic arc $\gamma'$.

 Moreover, since $\gamma'$ intersects two Fatou components,
 we have $q<p$.
\end{proof}

The following theorem follows by the proofs of \cite[Lemma~5.10,
Theorem~6.1, Theorem~6.3]{hubbard_multicorns_2014}:
\begin{thm}
 \label{thm:non-path connectivity}
 Let $\cH$ be the hyperbolic component of odd period and
 let $\cA \subset \partial \cH$ be the parabolic root arc.
 Let $c \in \cA$ be a parameter with zero critical Ecalle height.
 Let $T$ be its Hubbard tree for $f_c$,
 $x$ be the characteristic parabolic point and $\theta,\ \theta'$ be the
 landing angles for $x$.
 
 If the Hubbard loop is not equal to the equator, then 
 there exist $h>0$ and $\theta_n, \theta_n',
 \tilde{\theta}_n,\tilde{\theta}_n' \in \Q/\Z$ 
 such that
 \begin{itemize}
  \item $\theta_n ,\tilde{\theta}_n \to \theta$ and $\theta_n',
	\tilde{\theta}_n' \to \theta'$ .
  \item The dynamical rays $R_c(\theta_n)$ and $R_c(\theta_n')$
	land at $w_n$ of Ecalle height
	$h$, and $R_c(\tilde{\theta}_n)$ and
	$R_c(\tilde{\theta}_n')$ land at $\tilde{w}_n$ of Ecalle height
	$-h$. Both $w_n$ and $\tilde{w}_n$ are in the Hubbard tree and
	repelling preperiodic.
  \item There exists an arc $\cA' \subset \cA$ of positive length
	such that the limits of parameter rays
	$R_{\cM^*}(\tilde{\theta}_n)$ and $R_{\cM^*}(\theta_n')$
	contains $\cA'$.
 \end{itemize}
 In particular, consider any arc $\gamma$ in $\cM^*$ accumulating
 $\cA$, contained in the component of the complement of
 $R_{\cM^*}(\theta) \cup R_{\cM^*}(\theta') \cup \cA$ not
 containing $\cH$.
 Then $\gamma$ accumulates all points in $\cA'$ and does not
 converge to a point. 
\end{thm}

\begin{cor}
 \label{cor:non-path connectivity}
 Let $\cH$ be the hyperbolic component of odd period $p>1$ and
 let $\cA \subset \partial \cH$ be the parabolic root arc.
 Let $c \in \cA$ be the parameter with zero critical Ecalle height.
 Assume $f_c$ has a renormalization $g=f_c^n|_{U'}:U' \to U$ with maximal
 $n<p$.
 Let $T$ be a parabolic tree and $T' \subset T$ be a parabolic tree
 for the renormalization.
 
 If $T'$ contains a periodic point $y$ of non-real
 multiplier, then the conclusion of Theorem~\ref{thm:non-path
 connectivity} holds.
\end{cor}

Note that $y$ is repelling because the only non-repelling cycle is
the parabolic cycle, so it is of multiplier one, and that parabolic
trees are uniquely determined on the Julia set, hence the assumption
does not depend on the choice of parabolic trees.
See also the remark after \cite[Theorem~6.2]{hubbard_multicorns_2014}.

\begin{proof}
 Let $\gamma \subset T'$ be an subarc converging to the principal
 parabolic point $x$ contained in the domain of definition of the
 repelling Fatou coordinate at $x$. Then $\gamma$ corresponds to the
 Hubbard loop in the Ecalle cylinder.
 Observe that $\gamma$ intersects infinitely many (in particular two)
 Fatou components, since $\gamma \cap J_c$ is a Cantor set.
 
 Now assume that the Hubbard loop coincides with the equator. 
 Then we can apply Lemma~\ref{lem:analytic arc} to a subarc of
 $\gamma$.
 Let $\gamma'$ be the arc in the Lemma.
 Then by Corollary~\ref{cor:analytic tree} and the maximality of the
 period of the renormalization $g$, it follows that
 $y \in T' \cap J(f) \subset \gamma'$. 
 In particular, $y \in \bigcup_{n \ge 0}f_c^n(\gamma)$.
 However, by assumption, parabolic tree is not analytic at $y$,
 hence it is a contradiction.

 Therefore, we can apply Theorem~\ref{thm:non-path connectivity}.
\end{proof}

\subsection{Discontinuous straightening map}

Combining those results with Corollary~\ref{cor:real surjectivity}, we can
prove discontinuity of straightening maps as follows.
Let $\HAp \subset \Int \cM^*$ be the hyperbolic
component containing $\cAp$. Let $\cAp'=-1.75=\sup(\HAp \cap \R)$ be
the landing point of the umbilical cord for $\HAp$.

\begin{thm}
 \label{thm:discont}
 Let $c_0$ be a parameter with periodic critical point of odd period
 $p>1$. Let $\cH=\chi_{c_0}^{-1}(\HAp)$ be the hyperbolic
 component in $\cC(c_0)=\cR(c_0)$ corresponding to the airplane
 component. Let $b=\chi_{c_0}^{-1}(\cAp') \in \partial \cH$ be the
 unique parameter on the parabolic root arc with critical Ecalle height
 zero.

 If the (unique) period two cycle of a $c_0$-renormalization
 $g=f_{b}^p|_{U'}:U' \to U$ for $f_{b}$ has non-real multiplier,
 then $\chi_{c_0}^{-1}(c)$ converge to an arc of positive length
 in the parabolic root arc $\cA$ in $\partial \cH$ as $c \searrow
 \cAp'$.  In particular, $\chi_{c_0}$ is not continuous.
\end{thm}

Note that $\cAp'$ is the parameter on the parabolic root arc in
$\partial \HAp$ with critical Ecalle height zero, hence
$\chi_{c_0}(c_0') = \cAp'$ because hybrid conjugacy preserves the critical
Ecalle height. Also, recall that $\chi_{c_0}^{-1}(c') \in \cM^*$ is 
well-defined for $c' \in \R \cap \cM^*=[-2,1/4]$ by
Theorem~\ref{thm:IH-inj} and Corollary~\ref{cor:real surjectivity}.

\begin{proof}
 The Hubbard tree for $f_{\cAp'}$ is an interval 
 $[-1.75, 1.3125]$, which contains the unique period two cycle for
 $f_{\cAp'}$. Hence by assumption, we can apply Corollary~\ref{cor:non-path
 connectivity}.

 On the other hand, assume $c_n \searrow \cAp'$ and
 $b_n=\chi_{c_0}^{-1}(c_n) \to \tilde{b} \in \cC(c_0)$. 
 Let us denote $\tilde{c}=\chi_{c_0}(\tilde{b})$.
 Then $f_{\cAp'}$ and $f_{\tilde{c}}$ are quasiconformally conjugate.
 Hence $f_{b}$ and $f_{\tilde{b}}$ are also quasiconformally
 conjugate by \cite[Lemma~9.2]{MR2970463}.
 Therefore, $\tilde{b}$ lies in $\cA$ and it follows that
 $\chi_{c_0}^{-1}(c) \to \cA$ as $c \searrow \cAp'$.
 By Corollary~\ref{cor:non-path connectivity}, 
 $\chi_{c_0}^{-1}(c)$ has no limit as $c \searrow c_1'$.

 Since both $\cM^*$ and $\cC(c_0)$ are compact and $\chi_{c_0}$ is
 injective, $\chi_{c_0}$ is a homeomorphism onto its image 
 if it is continuous. Therefore $\chi_{c_0}$ is not continuous.
\end{proof}

By the theorem, it is easy to find discontinuous points of $\chi$ and
$\chi^{-1}$:

\begin{cor}
 Under the assumption of the Theorem, we have the following:
 \begin{itemize}
  \item $\chi_{c_0}^{-1}$ is not continuous at $\cAp'=-1.75$.
  \item $\chi_{c_0}$ is not continuous at any point in the parabolic root
	arc with sufficiently small but non-zero Ecalle height. More
	precisely, at any point in the accumulation set of the umbilical
	cord, i.e., any point in 
	\[
	 \bigcap_{\varepsilon>0}
	\overline{\chi_{c_0}^{-1}((\cAp',\cAp'+\varepsilon))},
	\]
	(possibly) except $c_1$.
 \end{itemize}
\end{cor}

We have not proved discontinuity of $\chi_{c_0}$ at
$b=\chi_{c_0}^{-1}(\cAp')$. 
However, decorations attached to the umbilical cord
also accumulate to the parabolic arc, hence
it is likely that $\chi_{c_0}$ is not continuous also at $b$.

Notice that by the symmetry of $\cM^*$, the same statement
holds for $\omega \cAp'$ and $\omega^2 \cAp'$ (equivalently,
we can apply Theorem~\ref{thm:discont} to the other straightening maps
$\omega\chi_{c_0}$ and $\omega^2\chi_{c_0}$ with different internal
markings).

\section{Rigorous numerical computations}
\label{sec:computations}

In this section, we consider the case $c_0=\cAp$ 
and show by rigorous numerical computation
that we can apply Theorem~\ref{thm:discont}, and it follows that
$\chi_{c_*}$ not continuous near $b=\chi_{\cAp}^{-1}(\omega^2 \cAp')$. 
Note that $\chi_{\cAp}^{-1}(\cAp') \in \R$ and the umbilical cord lands there.
Rigorous computation is done with the CAPD library 
(\url{http://capd.ii.uj.edu.pl/}).
The program can be downloaded from the author's web page
(\url{http://www.math.kyoto-u.ac.jp/~inou/}).

In the following, we naturally identify $\R^2 \cong \C$.
Let $R=[-1.73875,-1.73825]\times [0.01555,0.01605] \subset \C$.
\begin{lem}
 Let $U=(-0.3,0.3)^2$.
 For any $c \in R$,
 there exists an anti-holomorphic  quadratic-like mapping
 $g_c = f_c^3|_{U'_c}: U'_c \to U$.
\end{lem}

\begin{figure}
 \centering
 \includegraphics[width=12cm]{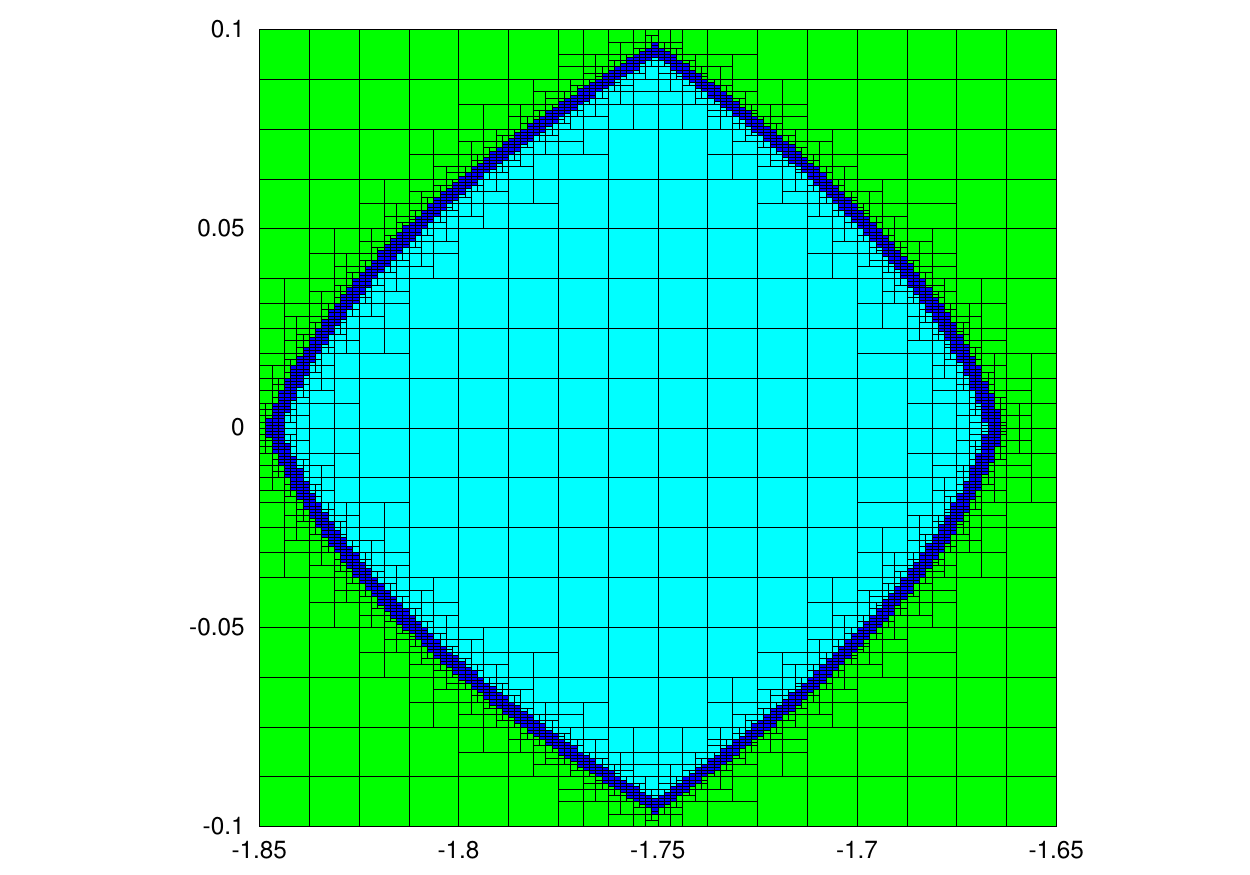}
 \includegraphics[width=12cm]{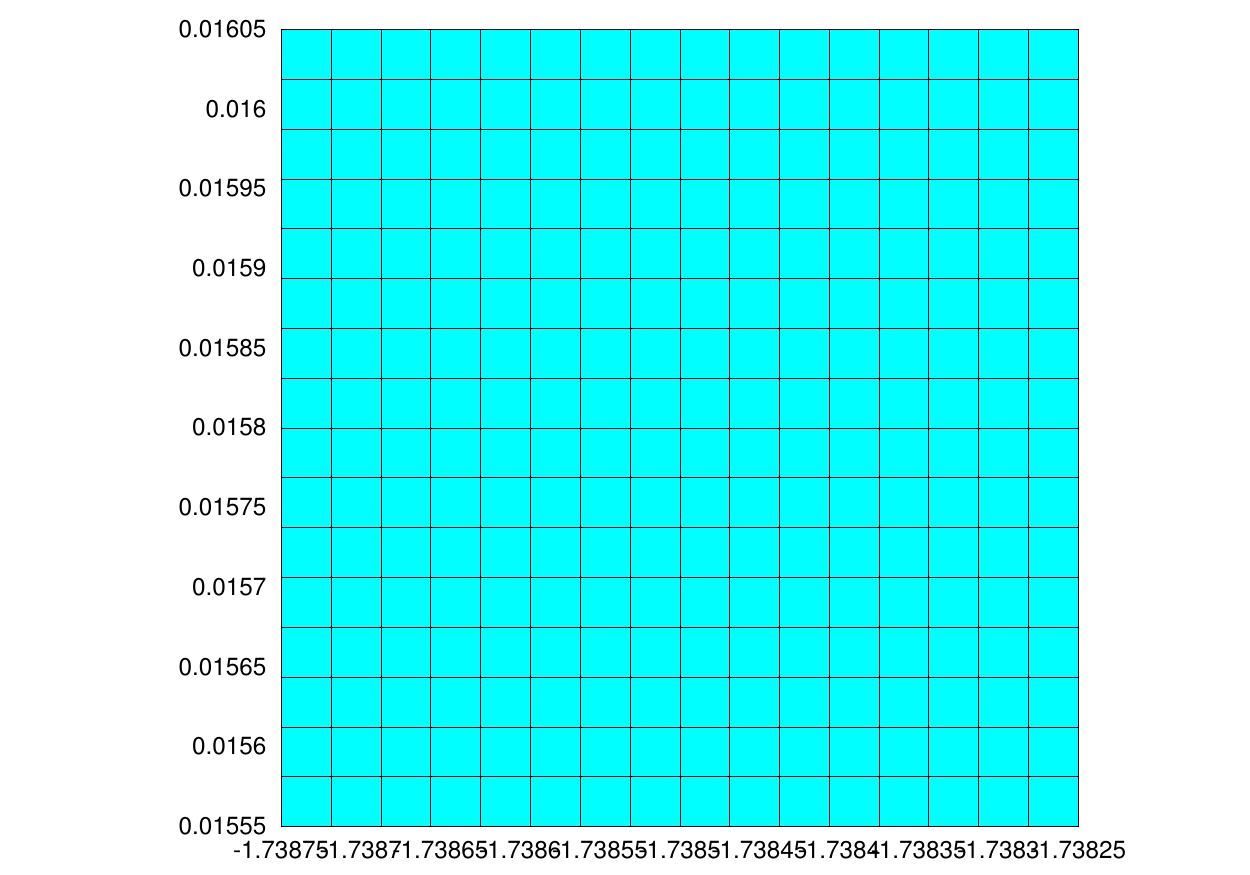}
 \caption{The existence of quadratic-like restriction of $f_c^3:U'_c \to
 U$. The existence is guaranteed in the cyan region.
 The second picture shows that the existence of quadratic-like
 restriction holds for all parameters in that region.}
 \label{fig:ren}
\end{figure}
Figure~\ref{fig:ren} shows the rigorous numerical result
for this lemma.

Precisely speaking, the program checks if the image of $\partial U$ for
a given rectangle $U$ intersects $\partial U$. 
Each box in Figure~\ref{fig:ren} is colored in green
if there is an intersection for any parameters, in cyan if there is no
intersection, in blue if undetermined.
Therefore, if a component $\cU$ of the cyan region contains a
renormalizable parameter,
then for all parameter $c \in \cU$, there exists a component $U_c'
\subset U$ of $f_c^{-3}(U)$ such that $f_c^3:U_c' \to U$ is a proper
map of degree two.

A lower and upper estimate of a hyperbolic component can be done as follows:
The argument principle allows us to verify that there exists a
unique periodic point of period $p$ in a given region (compare
Lemma~\ref{lem:unique per6}). 
Then it suffices to check that for any box $B$ in the region,
either $f^p(B)$ does not intersect $B$ or
$(f^p)'(B)$ is contained in the unit disk.

\begin{lem}
 There exists a unique hyperbolic component $\cH_9$ of period 9
 contained in $R \cap \cR(c_0)$. 
\end{lem}

\begin{figure}
 \centering
 \includegraphics[width=12cm]{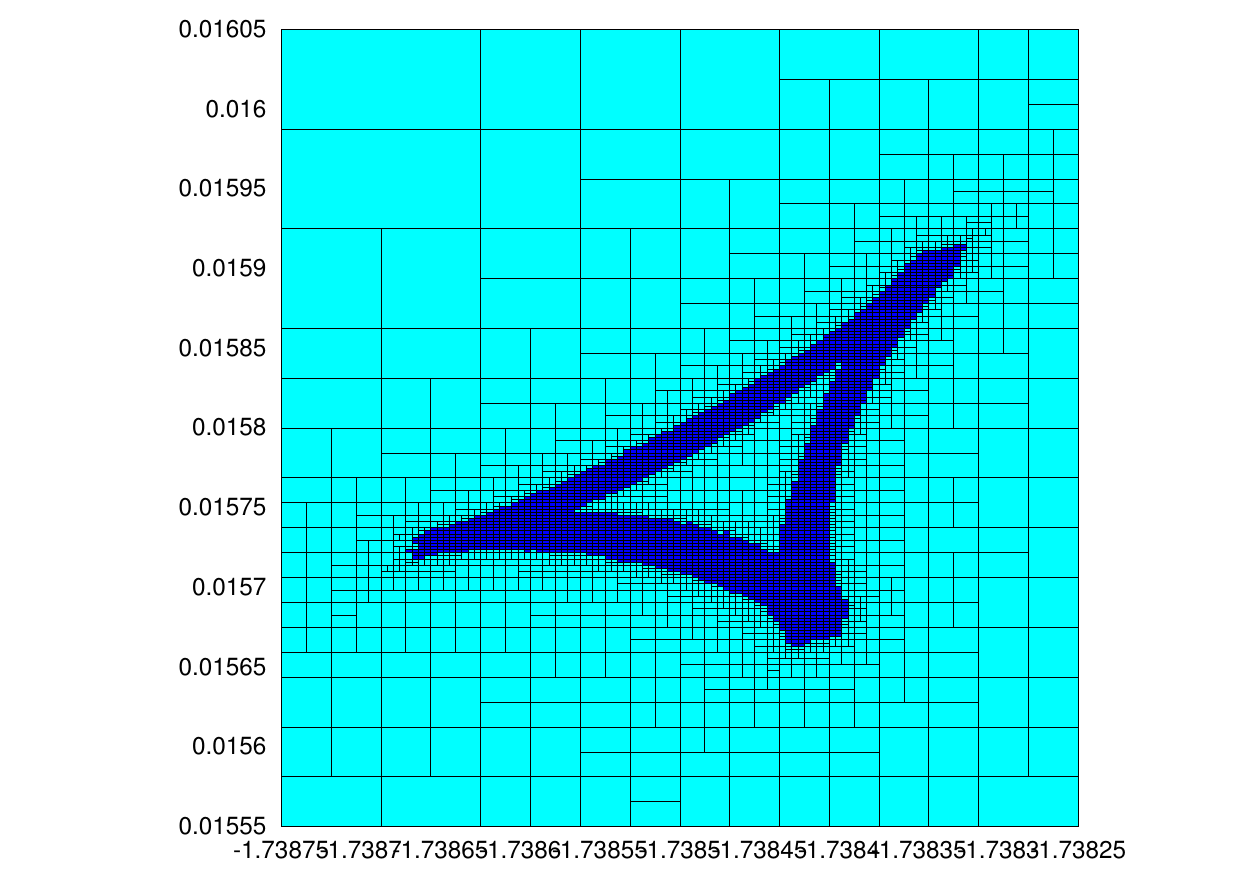}
 \caption{A period 9 hyperbolic component. Parabolic arcs are contained
 in the blue region.}
 \label{fig:per9}
\end{figure}

Figure~\ref{fig:per9} shows the position of $\cH_9$.
More precisely, there is no parabolic fixed point of multiplier one for
$f_c^9$ in the cyan region.
Therefore, it follows from the figure that the hyperbolic component $\cH_9$
is contained in the union of blue region and the bounded component of
the complement.

Precisely speaking, one also has to check that there is a period 9
attracting cycle for some parameter in the bounded cyan region, 
and there is no period 9 attracting cycle for some parameter in the
unbounded cyan region (and this is done rigorously for $\cH_9$).

\begin{lem}
 \label{lem:unique per6}
 For any $c \in R$,
 there exists a unique fixed point $x_c$ of $f_c^6$ 
 in $[0,0.08]^2$ contained in the filled Julia set of $g_c$.
\end{lem}

\begin{figure}
 \centering
 \includegraphics[width=12cm]{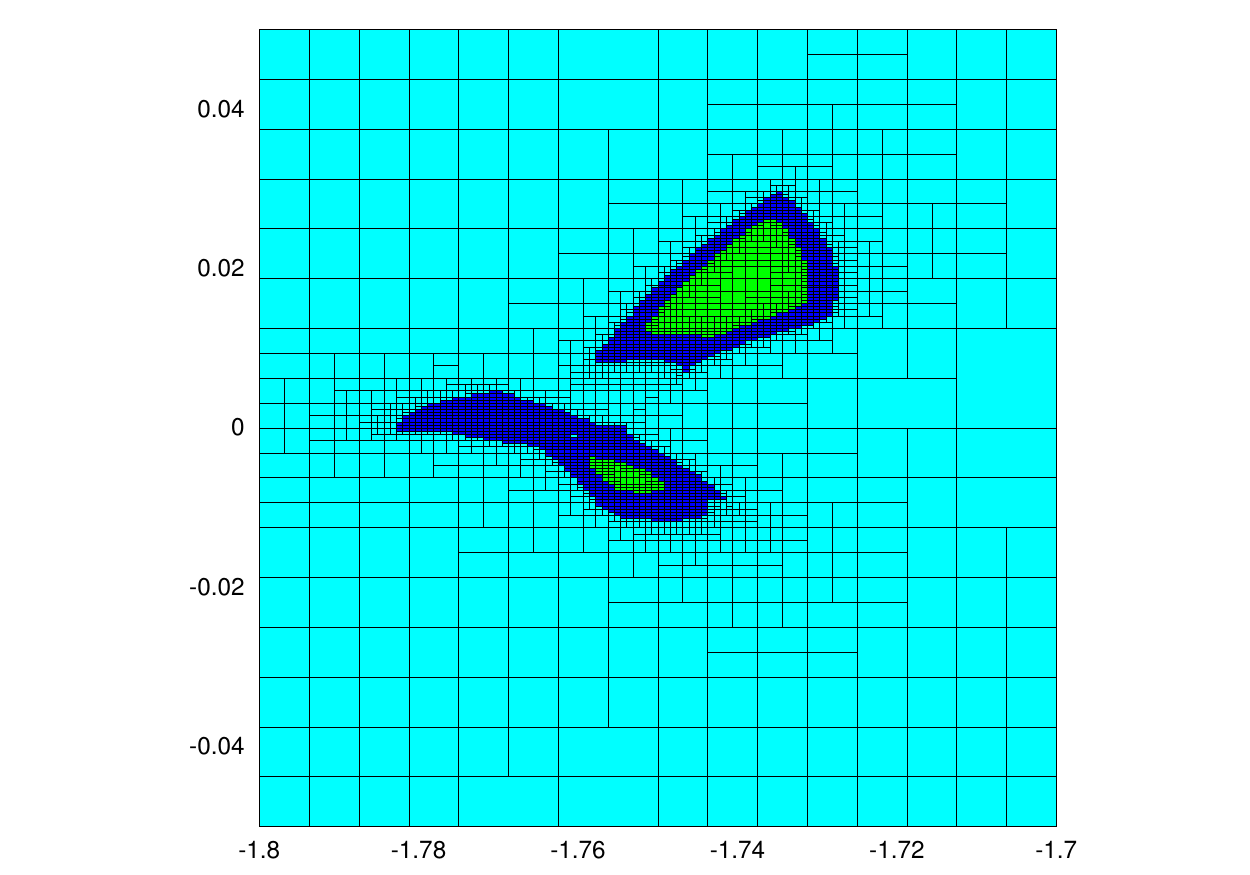}
 \includegraphics[width=12cm]{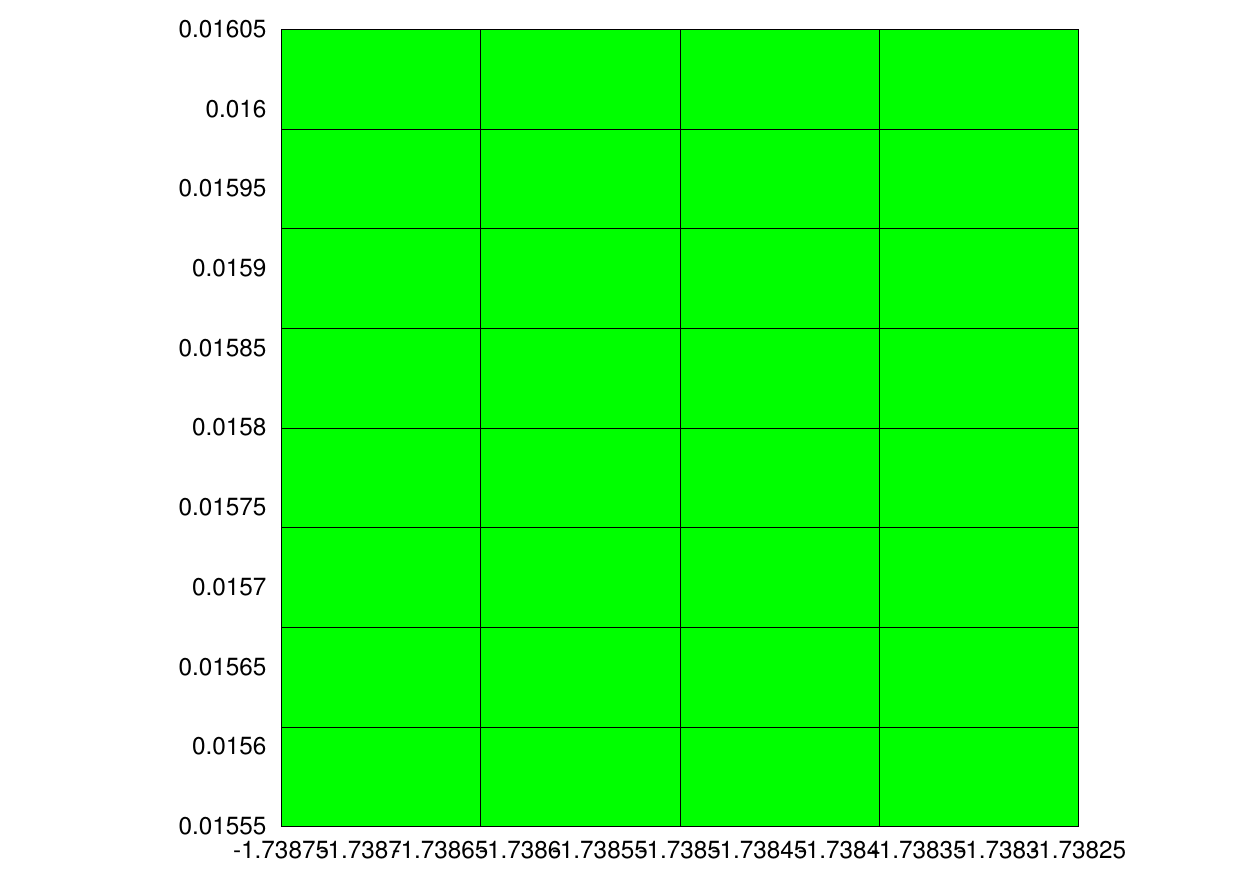}
 \caption{The region existing a unique fixed point in $[0,0.08]^2$.
 The unique existence holds in the green region,
 does not hold in the cyan region, and undetermined in the blue region.
 The second picture shows that the unique existence holds for all
 parameters in that region.} 
 \label{fig:uniqueperpt}
\end{figure}

Figure~\ref{fig:uniqueperpt} shows the rigorous numerical result for
this lemma.
The number of fixed points in a given region $R$ can be checked
by the argument principle.
Namely, it suffices to check if the rigorously estimated value of
\[
 \int_{\partial R} \frac{(f_c^6)'(z)-1}{f_c^6(z)-z}dz
\] 
contains $2\pi i$
and does not contain $0$ and $4\pi i$ (note that the computational
result is given by a box, hence it is convex).

As in the case of checking renormalizability for a specific parameter,
it is easy to check a condition on a multiplier of the unique periodic
point whose existence is guaranteed by the above lemma:
\begin{lem}
 \label{lem:disjoint}
 Let $C = \{c \in R;\ x_c \mbox{ has real multiplier for }f_c\}$.
 Then $C \cap \partial \cH_9 = \emptyset$.
\end{lem}

\begin{figure}
 \centering
 \includegraphics[width=12cm]{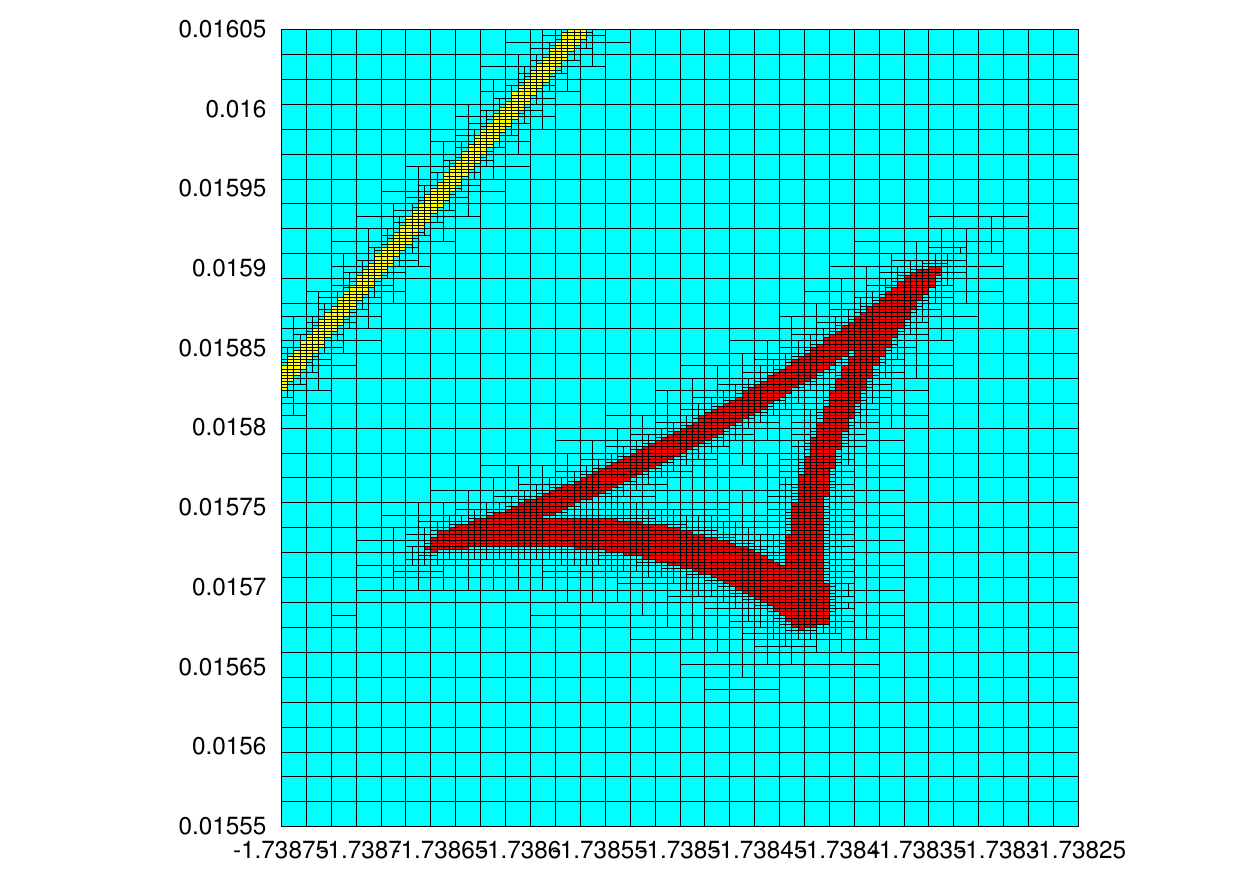}
 \caption{The real multiplier locus $C$ (contained in the yellow
 region) and the parabolic arcs $\partial \cH_9$ (contained in the red
 region).} 
 \label{fig:nobabytricorn}
\end{figure}

Figure~\ref{fig:nobabytricorn} shows that $C$ and $\partial \cH_9$ do
not intersect (compare Figure~\ref{fig:H_**}).

In particular,
the assumption of Theorem~\ref{thm:discont} is satisfied 
on the parabolic root arc $\cA$ in $\partial \cH_9$.
Therefore, we have the following.
\begin{cor}
 The straightening map $\chi_{\cAp}:\cC(\cAp) \to \cM^*$ is not
 continuous.

 More precisely, let $b$ be the unique parameter on the root arc $\cA
 \subset \partial \cH_9$ of critical Ecalle height zero.
 Then there exists a subarc $\cA' \subset \cA$ containing $b$ of
 positive length such that $\chi_{\cAp}$ is not continuous at any point
 in $\cA' \setminus \{b\}$, and $\chi_{\cAp}^{-1}$ is not continuous
 at $\chi_{\cAp}(b)=\omega^2 \cAp'$.
\end{cor}

By this corollary, it follows that at least ``half'' of baby
tricorn-like sets are not baby tricorns.
\begin{cor}
 For any baby tricorn-like set $\cC(c_0)$, let
 $c_1=\chi_{c_0}^{-1}(\cAp)$.
 Then either $\cC(c_0)$ or $\cC(c_1)$ is not a baby tricorn.
\end{cor}

\begin{proof}
 Let $L=\chi_{c_1}^{-1}(\omega\cdot(-1.75,0])$. 
 Then $L$ is the ``umbilical cord'' for the hyperbolic component $\cH$
 containing $c_2=\chi_{c_1}^{-1}(\omega \cAp)$ where
 $\omega=\frac{-1+\sqrt{3}i}{2}$ is a cubic root of unity.
 If it converge to a point, then $\chi_{c_0}$ is not a homeomorphism
 because the corresponding ``umbilical cord'' by $\chi_{c_0}$ does not
 converge by the previous corollary.

 Otherwise, $\chi_{c_1}$ is not a homeomorphism because
 the corresponding umbilical cord by $\chi_{c_1}$ is a real segment,
 which converges to $-1.75$.
\end{proof}

\appendix

\section{Straightening maps}
\label{sec:app-streightening}

Here we briefly describe how the notions in this paper corresponds to 
the notions in \cite{MR2970463} and how the results there
are applied to our case.

Let $g_c = f_c^2$.
This gives a natural embedding of the tricorn family $\C \cong
\{f_c\}_{c \in \C}$ into the family of monic centered quartic
polynomials $\poly(4)$ with 
connectedness locus $\cC(4)$.
For $c_0$ as above, $g_{c_0}$ is a hyperbolic post-critically hyperbolic
quartic polynomial with rational lamination
$\lambda(g_{c_0})=\lambda_0$.
The combinatorial renormalization locus $\cC(\lambda_0) = \{g \in
\poly(4);\ \lambda(g) \supset \lambda_0\}$ and 
$\cR(\lambda_0)=\{g \in \cC(\lambda_0);\ \lambda_0$-renormalizable$\}$
satisfies $\cC(c_0) = \cC(\lambda_0) \cap \C$ and
$\cR(c_0)=\cR(\lambda_0) \cap \C$.
Furthermore, the straightening map $\chi_{\lambda_0}:\cR(\lambda_0) \to
\cC(T(\lambda_0))$, where $\cC(T(\lambda_0))$ is the fiberwise
connectedness locus of the family of monic centered polynomial map over
the mapping schema $T(\lambda_0)$ of $\lambda_0$.

We describe what $\cC(T(\lambda_0))$ is in our case.
Let $c \in \cR(c_0)$.
When the renormalization period $n$ is even, 
the $c_0$-renormalization $f_c^n:U' \to U$ yields two
quadratic-like maps for $g_c$, i.e., $g_c^{n/2}:U' \to U$ and
$g_c^{n/2}:f_c(U') \to f_c(U)$ (by shrinking $U'$ and $U$ if
necessary). 
Observe that the second is anti-holomorphically conjugate to the first
by $f_c^{n-1}$ near the filled Julia set.
Therefore, as a $\lambda_0$-renormalization for $g_c$, we have 
two quadratic-like maps which are anti-holomorphically equivalent.
Therefore, the straightening of those are $P_{c'}$ and $P_{\bar{c'}}$,
where $c'=\chi_{c_0}(c)$.
In this case, $\cC(T(\lambda_0)) \cong \cM \times \cM$ 
and $\chi_{\lambda_0}(g_c) = (c',\bar{c'})$.
In other words, $\chi_{c_0}$ is the following composition:
\[
 \cR(c_0) \hookrightarrow \cR(\lambda_0)
 \xrightarrow{\chi_{\lambda_0}} \cC(T(\lambda_0))
 \supset \cC(T(\lambda_0)) \cap \{(c',\bar{c'});\ c' \in \C\}  
 \xrightarrow{\cong} \cM.
\]

When $n$ is odd, then $g_c^n:U' \to g_c^n(U')=f_c^n(U)$ is a
quartic-like map (similarly by shrinking $U'$ and $U$ if necessary).
More precisely, $g_c^n$ is divided into a composition of two degree two
maps, namely, 
$g_c^{\frac{n-1}{2}}:U' \to f_c^{n-1}(U')$ and 
$g_c^{\frac{n+1}{2}}:f_c^{n-1}(U') \to g_c^n(U')$.
Hence it follows that the straightening is a biquadratic polynomial
(i.e., a composition of two quadratic polynomials).
By definition, $\cC(T(\lambda_0))$ is the fiberwise connectedness locus
of the biquadratic family
\begin{align*}
 \poly(2\times 2) &=
 \{P:\{0,1\}\times \C \circlearrowleft;\ P(k,z) =
(1-k,P_k(z)),\ P_0,P_1 \in \poly(2)\} \\
 &\cong \{(Q_{c_0},Q_{c_1});\ c_0,c_1 \in \C\} \cong \C^2.
\end{align*}
We also identify $P \in \poly(2\times 2)$ with a quartic polynomial
$R_P(z)=Q_{c_1}\circ Q_{c_0}$ 
(the second iterate restricted to the first complex plane $\{0\} \times \C$).
Let $P =(P_0,P_1) = \chi_{\lambda_0}(g_c) \in \cC(T(\lambda_0))$.
Observe that
\[
 g_c^n:U' \to g_c^n(U'),\quad g_c^n:f_c^{n-1}(U') \to f_c^{2n-1}(U')
\]
are anti-holomorphically conjugate by $f_c^{n-1}$.
Hence it follows that their straightenings are anti-holomorphically
conjugate,
more precisely, we have $(P_1,P_0)= \bar{P}=(\bar{P_0},\bar{P_1})$.
Therefore, $P \in \{(P_0,\bar{P_0})\in \poly(2\times 2)\} \cong \C$.
When $P_0=Q_{\bar{c'}}$, then 
$P^2(0,z) = (0,(z^2+\bar{c'})^2+c')=(0,f_{c'}^2(z))$.
In other words, by the identification $(0,z) \sim (1,\bar{z})$ in
$\{0,1\} \times \C$, $P$ is
semiconjugate to $f_{c'}$ where $c'=\chi_{c_0}(c)$. Therefore in this
case, $\chi_{c_0}$ is the following composition:
\[
 \cR(c_0) \hookrightarrow \cR(\lambda_0) 
 \xrightarrow{\chi_{\lambda_0}} \cC(T(\lambda_0)) 
 \supset \cC(T(\lambda_0)) \cap \{(c',\bar{c'});\ c' \in \C\} 
 \cong \cM^*.
\]

With this in mind, Theorems~\ref{thm:IH-inj}, \ref{thm:IH-onto hyp},
\ref{thm:IH-cpt} and \ref{thm:IH-onto} 
are direct consequences of Theorem B, C, D and F in \cite{MR2970463}
respectively.

\section{Hubbard tree and renormalizability}
\label{sec:app-tree-renorm}

Here we discuss (combinatorial) renormalizability and existence
of an invariant Hubbard subtree.
Before describing the result, we recall 
the notion of \emph{fiber} \cite{MR2970463}:

\begin{defn}
 Let $c_0$ be a center and $\lambda=\lambda(f_{c_0})$.
 For $f \in \cC(c_0)$ and a $\lambda$-unlinked class $L$,
 define the \emph{fiber} $K_f(L)$ as follows:
 \[
  K_f(L) = K(f) \cap \bigcup_{\theta \sim_\lambda \theta',\ \theta\ne
 \theta'} \overline{\sector(\theta,\theta';L)},
 \]
 where $\sector(\theta,\theta';L)$ is a component
 of $\C \setminus (\overline{R_f(\theta) \cup R_f(\theta')})$ containing
 $R_f(\theta'')$ for $\theta'' \in L$.
\end{defn}

\begin{thm}
 \label{thm:tree-comb_renorm}
 Let $c \in \cM^*$ be principal parabolic and let $H \subset K(f_c)$ be
 its parabolic Hubbard tree.

 Assume there exist $q>0$ and
 a non-trivial subtree $H' \subset H$ such that
 \begin{itemize}
  \item $(f^j(H');\ j=1,\dots,q)$ are mutually disjoint.
  \item $0 \in H'$.
  \item For each $0 \le i, j< q$ with $i \ne j$,
	there exists $\theta, \theta' \in
	\Q/\Z$ such that $R_f(\theta)$ and $R_f(\theta')$ land at the
	same point and the closure of each component $\C \setminus
	\overline{R_f(\theta) \cup R_f(\theta')}$ contains exactly one
	of $f^i(H')$ and $f^j(H')$.
 \end{itemize}
 Then there exists a center $c_0 \in \cM^*$ such that $f \in
 \cC(\lambda(f_{c_0}))$ and the (unique) critical $\lambda$-unlinked
 class $L$ satisfies $H' \subset K_{f_c}(L)$.
 
\end{thm}

\begin{proof}
 Define an equivalence relation $\lambda \subset \lambda(f)$ as follows:
 Two angles $\theta, \theta' \in \Q/\Z$ are equivalent if 
 they are $\lambda(f)$-equivalent and for any $n \ge 0$,
 $\gamma_n = \overline{R_f(d^n\theta) \cup R_f(d^n\theta')}$ does not
 ``cut'' any $H_v$;
 more precisely, each $H_v$ is contained in the closure of some 
 component of $\C \setminus \gamma_n$.
 
 Since $f$ is parabolic, $\lambda(f)$ is hyperbolic, so there is no
 critical $\lambda(f)$-class. 
 In particular, for every $\lambda(f)$-class $A$, $m_d:A \to m_d(A)$ is 
 a cyclic order preserving bijection and at the landing point of the
 rays of angles in $A$, $f$ is locally a homeomorphism.
 Thus it is straightforward to show that $\lambda$ is a hyperbolic
 $d$-invariant rational lamination.
 Moreover, for each $v \in |T|$,
 there exists a $\lambda$-unlinked class $L_v$ such that the
 corresponding fiber $K_f(L_v)$ contains $H_v$.
\end{proof}

\begin{thm}
 \label{thm:tree-renorm}
 Let $c \in \cM^*$ be principal parabolic of odd period $p>1$.
 Assume there exist $q \ge 1$ and a subtree $H'$ of a parabolic tree $H$
 such that 
 \begin{itemize}
  \item $f_c^q(H')=H'$.
  \item $f_c^i(H') \cap f_c^j(H') = \emptyset$ for $0\le i,j < q$, $i\ne
	j$.
  \item $0 \in H'$ and $f_c(0) \subset f(H')$ contains the
	characteristic parabolic point.
 \end{itemize}
 Then there exists an anti-holomorphic renormalization $f_c^q:U' \to U$
 of period $q$ such that $H' \subset K(f_c^q:U' \to U)$ and $H'$ contains
 a parabolic tree of the renormalization. 

 In particular, the filled Julia set $K(f_c^q:U' \to U)$ is equal to the
 closure of the Fatou component containing $0$ if and only if $p=q$.
 Also, we have $H'=H$ if and only if $q=1$.
\end{thm}

\begin{proof}
 First observe that the assumption of the
 Theorem~\ref{thm:tree-comb_renorm} holds for
 $(f_c^j(H')$; $j=0,\dots,q-1)$;
 otherwise,
 $f_c^i(H')$ and $f_c^j(H')$ intersects the same Fatou component.
 Then both $f_c^i(H')$ and $f_c^j(H')$ contains the parabolic periodic point
 on the boundary by invariance, hence they intersect.
 
 By Theorem~\ref{thm:tree-comb_renorm}, 
 there exists a rational lamination $\lambda \subset \lambda(f_c)$ 
 such that the critical $\lambda$-unlinked class $L$ satisfies $H'
 \subset K_{f_c}(L)$ and $L, m_{-2}(L),\dots,m_{-2}^{q-1}(L)$ are
 mutually disjoint.
 
 Since $q|p$, $q$ is odd. Hence by Kiwi's lemma
 (Lemma~\ref{lem:Kiwi-realization}) and Lemma~\ref{lem:odd-primitive}, 
 $\lambda=\lambda(f_{c_0})$ for some $c_0$ and it is primitive. 
 By Theorem~\ref{thm:IH-cpt}, 
 we have $\cC(c_0)=\cR(c_0)$. 
 Thus $f_c$ is renormalizable of period $q$.
 Since $H',\dots,f_c^{q-1}(H')$ intersect all the periodic bounded Fatou
 components, parabolic cycle and the critical orbit, $H'$ contains a
 parabolic Hubbard tree.
\end{proof}

\bibliography{tricorn}

\end{document}